\documentclass{gen-j-l}
\usepackage [utf8]{inputenc}
\usepackage[T1]{fontenc}
\usepackage{amssymb}
\usepackage{etex}
\usepackage{pstricks}
\usepackage{pst-plot}
\usepackage{eso-pic}
\usepackage{graphicx}
\usepackage{tikz}
\usepackage{listings}
\usepackage{enumerate}
\usepackage[cmtip,all]{xy}
\copyrightinfo{}{}

\newtheorem{theo}{Theorem}[subsection]

\newtheorem{theoa}{Theorem}
\newtheorem{prop}[theo]{Proposition}
\newtheorem{cor}{Corollary}[theo]

\newtheorem{defi}[theo]{Definition}

\newtheorem{remark}[theo]{Remark}

\newcommand{\Flo}{\mathit{J}}

\newcommand{\Op}{\mathrm{Op}}

\newcommand{\Vois}{\mathrm{Neigh}}
\newcommand{\dist}{\mathrm{dist}}
\newcommand{\tub}{\mathrm{TubNeigh}}

\begin{document}

\title[Bohr-Sommerfeld conditions]{Bohr-Sommerfeld quantization conditions for non-selfadjoint perturbations of selfadjoint operators in dimension one.}

\author{Ophélie Rouby}
\address{IRMAR (UMR 6625 du CNRS), Universit\'e de Rennes 1, campus de Beaulieu. B\^at. 22-23. 35042 Rennes Cedex}
\curraddr{}
\email{ophelie.rouby@univ-rennes1.fr}
\thanks{}

%    \subjclass is required by all journals except JAG.
\subjclass[1991]{}

\date{}

\dedicatory{}

\begin{abstract}
In this paper, we give a description of the spectrum of a class of non-selfadjoint perturbations of selfadjoint $\hbar$-pseudo-differential operators in dimension one and we show that it is given by Bohr-Sommerfeld quantization conditions. To achieve this, we make use of previous work by Michael Hitrik, Anders Melin and Johannes Sjöstrand. We also give an application of our result in the case of $ \mathcal{P} \mathcal{T}$-symmetric pseudo-differential operators.
\end{abstract}

\maketitle
\setcounter{tocdepth}{1}

\section*{Introduction}
The object of this paper is to describe the spectrum of a class of pseudo-differential operators in the semi-classical limit. Semi-classical analysis is a rigorous mathematical framework that allows to relate classical mechanics and quantum mechanics in the regime where the Planck constant $ \hbar$ goes to zero, using the microlocal analysis of pseudo-differential operators as its basic tools. In quantum mechanics to each observable we associate an operator and the possible values of this observable correspond to the spectrum of the operator. Most of the papers on quantum mechanics are devoted to the study of selfadjoint operators which corresponds to real observables. However, the spectral analysis of non-selfadjoint operators, which is
the subject of this paper, is very useful in the modelling of damping or for the study of resonances \cite{MR1912874}. Moreover, PT-symmetric operators, which are not necessarily selfadjoint but may have, under some conditions, a real spectrum, have been recently considered as a natural generalization of quantum observables, see \cite{Bender}. \\

Bohr-Sommerfeld quantization conditions were introduced in the study of electronic levels of atoms to determine which classical trajectories were relevant. More precisely, Niels Bohr proposed that the electrons in atoms could only exist in certain well-defined stable orbits satisfying the following condition:
$$ \dfrac{1}{2 \pi} \oint p.dq = n \hbar , \quad \text{for some $n \in \mathbb{N}$},$$
where the pair $(q,p)$ are the position and momentum coordinates of an electron and where the integral is computed over some closed orbit in the phase space. Further experiments showed that Bohr's model of the atom seemed too simple to describe some heavier elements, so Arnold Sommerfeld expanded the original model to explain these phenomenons by
suggesting that electrons travel in elliptical orbits around a nucleus instead of circular orbits. \\

Mathematically, Bohr-Sommerfeld quantization conditions give a description of the spectrum of some class of selfadjoint operators. These conditions are established in the one-dimensional case and in that of completely integrable systems. More precisely, we say that a one-dimensional selfadjoint pseudo-differential operator $P_{ \hbar}$ satisfies the Bohr-Sommerfeld conditions if its eigenvalues are the real numbers $E$ such that:
$$ \int_{ \gamma_E} \alpha_0 + \hbar \int_{ \gamma_E} \kappa + \hbar \dfrac{\pi}{2} \mu( \gamma_E) + \mathcal{O}( \hbar^2) \in 2 \pi \hbar \mathbb{Z},$$
where:
\begin{itemize}
\item $ \gamma_E$ is a specific loop in the level set $ \Lambda_E = \lbrace p^{-1}(E) \rbrace$ where $p$ is the principal symbol of the operator $P_{ \hbar}$;
\item $ \mu( \gamma_E)$ is the Maslov index of the curve $ \gamma_E$;
\item $ \alpha_0$ is the Liouville $1$-form;
\item $ \kappa$ is the subprincipal $1$-form.
\end{itemize}
The case of regular energy curves in dimension one has been investigated by Bernard Helffer and Didier Robert in \cite{MR776281} and that of completely integrable systems by Anne-Marie Charbonnel in \cite{MR962310} and by San V\~u Ng\d{o}c in \cite{MR1721373} (where the subprincipal $1$-form was defined). In the case of non-selfadjoint operators, these conditions are not satisfied; however if we consider a non-selfadjoint pseudo-differential operator close to a selfadjoint one, several results have been recently obtained. More precisely, these conditions have been extended in the case of non-selfadjoint perturbations of selfadjoint operators in dimension two by Anders Melin and Johannes Sjöstrand in \cite{MR1957486,MR2003421} and then by Michael Hitrik and Johannes Sjöstrand in \cite{MR2036816}. The case of non-selfadjoint perturbations of selfadjoint operators in dimension one has not yet been treated, so that is the case that we investigate here. \\
More precisely, we give a description of the spectrum of a family of $\hbar$-pseudo-differential operators of the form $F^{ \epsilon}_{\hbar}(x, \hbar D_x) + i \epsilon Q^{ \epsilon}_{\hbar}(x, \hbar D_x)$ where $F_{ \hbar}^{ \epsilon}$ and $Q_{ \hbar}^{ \epsilon}$ are selfadjoint operators depending smoothly on a parameter $ \epsilon$. The result states that any eigenvalue of such object can be written as a function of $\hbar$ times an integer. This function is analytic, depending on the small parameter $ \epsilon$ and admits an asymptotic expansion in powers of $\hbar$. Moreover, the first term in the asymptotic expansion of this function is the inverse of a complex action integral. Then, we give an application of our result for $ \mathcal{P} \mathcal{T}$-symmetric pseudo-differential operators in dimension one. \\
\\
\textit{Structure of the paper:}
\begin{itemize}
\item in Section 1, we state our result;
\item in Section 2, we prove the result in two main steps. The first one consists in establishing the result in the case of an operator acting on $L^2( \mathbb{S}^1)$ and to prove it by using techniques developed by Michael Hitrik, Anders Melin, Johannes Sjöstrand and San V\~u Ng\d{o}c in the following papers \cite{MR1957486,MR2003421,MR2036816,MR2288739,SJ}. More precisely, we use complex microlocal analysis and Grushin problems. Afterwards, in the second step, we generalize the result obtained in the first step to obtain our result;
\item in Section 3, we give an application of our result for $ \mathcal{P} \mathcal{T}$-symmetric operators;
\item in Section 4, we give some numerical illustrations of our result.
\end{itemize}

\noindent \textbf{Acknowledgements.} The author would like to thank San V\~u Ng\d{o}c for his advice and supervision, as well as Johannes Sjöstrand for kindly pointing out the application to $ \mathcal{P} \mathcal{T}$-symmetric operators. Funding was provided by the Universit\'e de Rennes 1 and the Centre Henri Lebesgue.

\section{Result}
\subsection{Assumptions}
Let $0 < \hbar \leq 1$ be the semi-classical parameter. Recall the definition of the Weyl quantization.

\begin{defi}
Let $p_{ \hbar}(x, \xi) \in \mathcal{S}( \mathbb{R}^2)$ be a function in the Schwartz space defined on the cotangent space $T^* \mathbb{R} \simeq \mathbb{R}^2$ and admitting an asymptotic expansion in powers of $ \hbar$. We define the Weyl quantization of the symbol $p_{ \hbar}$, denoted by $P_{ \hbar}(x, \hbar D_x)$ (where $D_x = -i \partial_x$), by the following formula, for $u \in L^2( \mathbb{R})$:
$$ P_{ \hbar}(x, \hbar D_x) u(x) = \dfrac{1}{2 \pi \hbar} \int_{ \mathbb{R}} \int_{ \mathbb{R}} e^{(i/ \hbar)(x-y) \xi} p_{ \hbar} \left( \dfrac{x+y}{2}, \xi \right) u(y) dy d \xi .$$
$P_{ \hbar}(x, \hbar D_x)$ is a pseudo-differential operator acting on $L^2( \mathbb{R})$ and we called the function $p_{ \hbar}(x, \xi)$ the symbol of the operator $P_{ \hbar}(x, \hbar D_x)$.
\end{defi}

\noindent Let $ \epsilon$ be a positive real number. Let $P^{ \epsilon}_{\hbar} = P^{ \epsilon}_{\hbar}(x, \hbar D_x) $ be the Weyl quantization on $ \mathbb{R}^2$ of some symbol $p^{ \epsilon}_{ \hbar}:= f^{ \epsilon}_{ \hbar} + i \epsilon q^{ \epsilon}_{\hbar}$ depending smoothly on $ \epsilon$ and satisfying:
\begin{enumerate}
\item[(A)] $p^{ \epsilon}_{ \hbar}$ is a holomorphic function on a tubular neighbourhood of $ \mathbb{R} \times \mathbb{R}$ and on this tubular neighbourhood we have:
\begin{equation} \label{eq_borne_fonction_ordre}
\exists C >0, \quad | p^{ \epsilon}_{ \hbar}(x, \xi) | \leq C m( \Re(x, \xi)),
\end{equation}
where $m$ is an order function on $ \mathbb{R}^2$, \textit{i.e.}
\begin{enumerate}
\item[1.] $m \geq 1$;
\item[2.] there exists some constants $C_0 \geq 0$ and $N_0 \geq 0$ such that, for all $ X, \hat{X} \in \mathbb{R}^2$:
$$ m(X) \leq C_0 \langle X- \hat{X} \rangle^{N_0} m( \hat{X}),$$
where $ \langle X \rangle = (1+|X|^2)^{1/2}$.
\end{enumerate}
\item[(B)] $p^{ \epsilon}_{ \hbar}$ admits an asymptotic expansion in powers of $\hbar$ in the space of holomorphic functions satisfying the bound \eqref{eq_borne_fonction_ordre} of the form:
$$ p^{ \epsilon}_{ \hbar}( x, \xi) \sim \sum_{j=0}^{ \infty} p^{ \epsilon}_j( x, \xi) \hbar^j ;$$
\item[(C)] the principal symbol, denoted by $p^{ \epsilon}$:
$$ p^{ \epsilon}( x, \xi) := p^{ \epsilon}_0(x, \xi) = f^{ \epsilon}(x, \xi) + i \epsilon q^{ \epsilon}(x, \xi),$$
with $ (x, \xi) \in \mathbb{R}^2$, is elliptic at infinity, \textit{i.e.} for $(x, \xi)$ in a tubular neighbourhood of $ \mathbb{R}^2$, there exists $C>0$ such that:
$$ \left| p^{ \epsilon}(x, \xi) \right| \geq \dfrac{1}{C} m( \Re(x, \xi)), \quad \text{for $|(x, \xi)| \geq C$} ;$$
\item[(D)] the symbols $f^{ \epsilon}_{ \hbar}$ and $q^{ \epsilon}_{ \hbar}$ are $ \mathbb{R}$-valued analytic functions on $ \mathbb{R}^2$ depending smoothly on $ \epsilon$.
\end{enumerate}

\noindent Therefore, we consider a pseudo-differential operator $P^{ \epsilon}_{ \hbar}$ acting on $L^2( \mathbb{R})$ satisfying the previous hypotheses, so we have:
$$ P^{ \epsilon}_{ \hbar}(x, \hbar D_x) = F^{ \epsilon}_{ \hbar}(x, \hbar D_x) + i \epsilon Q^{ \epsilon}_{ \hbar}(x, \hbar D_x) ,$$
where $F^{ \epsilon}_{ \hbar}$ and $Q^{ \epsilon}_{ \hbar}$ are selfadjoint pseudo-differential operators depending smoothly on $ \epsilon$. \\
In order to use the action-angle coordinates theorem, we consider, for $E_0 \in \mathbb{R}$ a fixed real number, the level set:
$$ \Lambda_{E_0} = \lbrace (x, \xi) \in \mathbb{R}^2; \left. p^{ \epsilon}(x, \xi) \right|_{ \epsilon = 0}=E_0 \rbrace.$$
We assume that:
\begin{enumerate}
\item[(E)] $ \Lambda_{E_0}$ is compact, connected and regular, \textit{i.e.} $\left. d (p^{ \epsilon} \right|_{ \epsilon = 0}) = \left. d (f^{ \epsilon} \right|_{ \epsilon = 0}) \neq 0$ on $ \Lambda_{E_0}$.
\end{enumerate}

\begin{remark}
Because of the ellipticity assumption, we already know that the level set $ \Lambda_{E_0}$ is compact for small $E_0$.
\end{remark}

\noindent \texttt{Notation:} $\tub(\mathbb{R}^2)$ denotes a tubular neighbourhood of $\mathbb{R}^2$ in $ \mathbb{C}^2$. \\

\noindent Assume, for $C>0$ a constant and for $ \epsilon_0>0$ a sufficiently small fixed real number, that:
$$ E \in \left\lbrace z \in \mathbb{C}; | \Re(z) -E_0 | < \dfrac{1}{C}, | \Im(z)| < \dfrac{\epsilon_0}{C} \right\rbrace.$$
We consider the following complex neighbourhood of the level set $ \Lambda_{E_0}$:
$$ \Lambda_E^{ \epsilon} = \lbrace ( x, \xi) \in \tub(\mathbb{R}^2); p^{ \epsilon}( x, \xi) = E \rbrace .$$ 
This level set is connected and $df^{ \epsilon} \neq 0$ on $ \Lambda_E^{ \epsilon}$ for $ \epsilon$ small enough (according to Assumption (E)). In what follows, we define an action integral $I(E)$ of the form:
$$ I(E) = \dfrac{1}{2\pi} \int_{ \gamma_E} \xi dx, $$
where $ \gamma_E$ is a specific loop in the level set $\Lambda_E^{ \epsilon}$ (see Section 2.3). We will show that under our assumptions the map $E \longmapsto I(E)$ is invertible. \\

\noindent Under the assumptions (A) to (D), the spectrum of the operator $P^{ \epsilon}_{ \hbar}$ is discrete in some fixed neighbourhood of the real number $E_0$.

\subsection{Main result}

\begin{theoa} \label{theo_result}
Let $P^{ \epsilon}_{ \hbar}$ be a pseudo-differential operator depending smoothly on a small parameter $ \epsilon$ and acting on $L^2( \mathbb{R})$. Let $E_0 \in \mathbb{R}$ such that the assumptions (A) to (E) are satisfied, in particular the operator $P_{ \hbar}^{ \epsilon}$ is of the form:
$$ P^{ \epsilon}_{ \hbar}(x, \hbar D_x) = F^{ \epsilon}_{ \hbar}(x, \hbar D_x) + i \epsilon Q^{ \epsilon}_{ \hbar}(x, \hbar D_x) ,$$
where $F_{ \hbar}^{ \epsilon}$ and $Q_{ \hbar}^{\epsilon}$ are selfadjoint pseudo-differential operators.
Let, for $ \epsilon_0>0$ a sufficiently small fixed real number:
$$ R_{C, \epsilon_0} = \left\lbrace z \in \mathbb{C}; | \Re (z) -E_0| <  \dfrac{1}{C}, | \Im (z) | < \dfrac{\epsilon_0}{C} \right\rbrace \quad \text{where $C > 0$ is a constant} .$$
Then the spectrum of the operator $P^{ \epsilon}_{ \hbar}$ in the rectangle $R_{C, \epsilon_0}$ is given by, for $ 0 \leq \epsilon < \epsilon_0$:
$$ \sigma(P^{ \epsilon}_{ \hbar}) \cap R_{C, \epsilon_0} = \lbrace g^{ \epsilon}_{ \hbar}(\hbar k) , k \in \mathbb{Z} \rbrace \cap R_{C, \epsilon_0} + \mathcal{O}(\hbar^{ \infty}) ,$$
where $g^{ \epsilon}_{ \hbar}$ is an analytic function admitting an asymptotic expansion in powers of $\hbar$ and depending smoothly on $ \epsilon$. Moreover, the first term in the asymptotic expansion of $g^{ \epsilon}_{ \hbar}$, denoted by $g^{ \epsilon}_0$, is the inverse of the action coordinate $ E \longmapsto \frac{1}{2 \pi} \int_{ \gamma_E} \xi dx$.
\end{theoa}

\begin{remark}
This result gives a description of the spectrum in a rectangle $R_{C, \epsilon_0}$ which does not depend on the semi-classical parameter $ \hbar$ contrary to the result obtained in the two-dimensional case by Michael Hitrik and Johannes Sjöstrand in \cite{MR2036816} in which the parameters $ \epsilon$ and $ \hbar$ are related. Therefore, we obtain a slightly finer result in the one-dimensional case.
\end{remark}

\begin{remark}
We assume that the level set $ \Lambda_{E_0}$ is connected. However, it should be possible to state a similar result in the case of several connected components using the same basic outline (in this case, we would have to consider Bohr-Sommerfeld quantization conditions for each component and consider the union of these components).
\end{remark}

\section{Proof}
\noindent The proof of our result is divided into two parts:
\begin{enumerate}[1.]
\item we consider a pseudo-differential operator $P^{ \epsilon}_{ \hbar}$ acting on $L^2( \mathbb{S}^1)$ of the form $P^{ \epsilon}_{ \hbar}( \theta, \hbar D_{ \theta})= F^{ \epsilon}_{ \hbar}( \theta, \hbar D_{ \theta}) + i \epsilon Q^{ \epsilon}_{ \hbar}( \theta, \hbar D_{ \theta})$, where $F_{ \hbar}^{ \epsilon}( \theta, \hbar D_{ \theta}) = F^{ \epsilon}( \hbar D_{ \theta}) + \mathcal{O}(\hbar)$ and we prove the same type of result (Theorem \ref{theo_cas_simple}) for this operator;
\item we generalize Theorem \ref{theo_cas_simple} to the case of an operator acting on $L^2( \mathbb{R})$ and satisfying the assumptions (A) to (E).
\end{enumerate}

\subsection{Result in the $L^2( \mathbb{S}^1)$-case}
In this paragraph, we present our result in the case of a pseudo-differential operator acting on $L^2( \mathbb{S}^1)$.\\

\noindent \texttt{Notation:} 
\begin{itemize}
\item $ \mathbb{S}^1$ is the real torus $\mathbb{R} / 2 \pi \mathbb{Z}$;
\item $(T^* \mathbb{S}^1)_{ \mathbb{C}}$ is the complex cotangent space of $ \mathbb{S}^1$: $ (\mathbb{S}^1 + i \mathbb{R}) \times \mathbb{C}$;
\item $L^2( \mathbb{S}^1)$ is the set of $2 \pi$-periodic measurable functions $f$ such that:
$$ \dfrac{1}{2 \pi} \int_0^{ 2 \pi}  \left| f( \theta) \right|^2 d \theta < \infty ;$$
\item $\tub(\mathbb{S}^1 \times \mathbb{R})$ is a tubular neighbourhood of $ \mathbb{S}^1 \times \mathbb{R}$ in $(T^* \mathbb{S}^1)_{ \mathbb{C}}$;
\item $ \Vois(A;B)$ is a neighbourhood of the space $A$ in the space $B$.
\end{itemize}

\noindent We consider a pseudo-differential operator $P^{ \epsilon}_{ \hbar}$ depending smoothly on $ \epsilon$ and acting on $L^2( \mathbb{S}^1)$ of the form:
$$ P^{ \epsilon}_{ \hbar}( \theta, \hbar D_{ \theta}) = F^{ \epsilon}_{ \hbar}( \theta, \hbar D_{ \theta}) + i \epsilon Q^{ \epsilon}_{ \hbar}( \theta, \hbar D_{ \theta}),$$
where $Q^{ \epsilon}_{ \hbar}$ is a selfadjoint pseudo-differential operator depending smoothly on $ \epsilon$ and $F^{ \epsilon}_{ \hbar}$ is a selfadjoint pseudo-differential operator depending smoothly on $\epsilon$ of the form:
$$ F^{ \epsilon}_{ \hbar}( \theta, \hbar D_{ \theta}) = F^{ \epsilon}(\hbar D_{ \theta}) + \mathcal{O}(\hbar) .$$
More precisely, $P^{ \epsilon}_{ \hbar}$ is the Weyl quantization of the symbol $p^{ \epsilon}_{ \hbar}( \theta, I) = f^{ \epsilon}_{ \hbar}( \theta, I) + i \epsilon q^{ \epsilon}_{ \hbar}( \theta, I)$ which is a periodic function in $ \theta$ satisfying the following conditions:
\begin{enumerate}
\item[(A')] $p^{ \epsilon}_{ \hbar}$ is a holomorphic function on a tubular neighbourhood of $ \mathbb{S}^1 \times \mathbb{R}$ such that on this neighbourhood:
\begin{equation} \label{eq_fonction_ordre_2}
\exists C >0, \quad |p^{ \epsilon}_{ \hbar}( \theta, I) | \leq C m( \Re(I)),
\end{equation}
where $m$ is an order function on $ \mathbb{R}$;
\item[(B')] $p^{ \epsilon}_{ \hbar}$ admits an asymptotic expansion in powers of $\hbar$ in the space of holomorphic functions satisfying the bound \eqref{eq_fonction_ordre_2}:
$$ p^{ \epsilon}_{ \hbar}( \theta, I) \sim \sum_{j=0}^{ \infty} p^{ \epsilon}_j( \theta, I) \hbar^j,$$
\item[(C')] the principal symbol, denoted by $p^{ \epsilon}$:
$$ p^{ \epsilon}( \theta, I) := p^{ \epsilon}_0( \theta, I) = f^{ \epsilon}(I) + i \epsilon q^{ \epsilon}( \theta, I), $$
for $( \theta, I) \in \mathbb{S}^1 \times \mathbb{R}$, is elliptic at infinity, \textit{i.e.} for $( \theta, I)$ in a tubular neighbourhood of $ \mathbb{S}^1 \times \mathbb{R}$, there exists $C>0$ such that:
$$ |p^{ \epsilon}( \theta, I) | \geq \dfrac{1}{C} m( \Re(I)), \quad \text{ for $| ( \theta, I)| \geq C$;}$$
\item[(D')] the symbols $f^{ \epsilon}_{ \hbar}$ and $q^{ \epsilon}_{ \hbar}$ are $ \mathbb{R}$-valued analytic functions on $ \mathbb{S}^1 \times \mathbb{R}$ depending smoothly on $\epsilon$.
\end{enumerate}
For $E_0 \in \mathbb{R}$ a fixed real number, we consider the level set:
$$ \Lambda_{E_0} = \lbrace (\theta, I) \in \mathbb{S}^1 \times \mathbb{R}; \left. p^{ \epsilon}( \theta, I) \right|_{ \epsilon = 0}=E_0 \rbrace.$$
We assume that:
\begin{enumerate}
\item[(E')] $ \Lambda_{E_0}$ is regular, \textit{i.e.} $  \left. d(p^{ \epsilon} \right|_{ \epsilon=0}) =  (\left. f^{ \epsilon} \right|_{ \epsilon = 0})' \neq 0$ on $ \Lambda_{E_0}$.
\end{enumerate}

\noindent Assume, for $C>0$ a constant and for $\epsilon_0>0$ a sufficiently small fixed real number, that:
$$ E \in \left\lbrace z \in \mathbb{C}; | \Re(z) -E_0 | < \dfrac{1}{C}, | \Im(z)| < \dfrac{\epsilon_0}{C} \right\rbrace.$$
We consider the following complex neighbourhood of the level set $ \Lambda_{E_0}$:
$$ \Lambda_E^{ \epsilon} = \lbrace ( \theta, I) \in \tub( \mathbb{S}^1 \times \mathbb{R}); p^{ \epsilon}( \theta, I) = E \rbrace .$$
According to Assumption (E'), this level set $ \Lambda_E^{ \epsilon}$ is connected and $(f^{ \epsilon})' \neq 0$ on $ \Lambda_E^{ \epsilon}$ for $ \epsilon$ small enough ($f^{ \epsilon}$ is a local diffeomorphism on $ \Lambda_E^{ \epsilon}$). \\
Let $ \gamma_E$ be a loop in $ \Lambda_E^{ \epsilon}$ generating $ \pi_1( \Lambda_E^{ \epsilon})$ (the fundamental group of $ \Lambda_E^{ \epsilon}$), we define an action integral $ \tilde{I}$ (we will explain later why this integral is well-defined and invertible) by:
$$ \tilde{I}(E) = \dfrac{1}{2\pi} \int_{ \gamma_E} I d \theta. $$

\noindent To describe the spectrum of the operator $P^{ \epsilon}_{ \hbar}$, we have the following result.

\begin{theoa} \label{theo_cas_simple}
Let $P^{ \epsilon}_{ \hbar}$ be a pseudo-differential operator depending smoothly on a small parameter $ \epsilon$ and acting on $L^2( \mathbb{S}^1)$. Let $E_0 \in \mathbb{R}$ such that the hypotheses (A') to (E') are satisfied, in particular the operator $P_{ \hbar}^{ \epsilon}$ is of the form:
$$ P^{ \epsilon}_{ \hbar}( \theta, \hbar D_{ \theta}) = F^{ \epsilon}_{ \hbar}( \theta, \hbar D_{ \theta}) + i \epsilon Q^{ \epsilon}_{ \hbar}( \theta, \hbar D_{ \theta}),$$
where $F_{ \hbar}^{ \epsilon}( \theta, \hbar D_{ \theta}) = F^{ \epsilon}( \hbar D_{ \theta}) + \mathcal{O}( \hbar)$. Let, for $ \epsilon_0>0$ a sufficiently small fixed real number:
$$ R_{C, \epsilon_0} = \left\lbrace z \in \mathbb{C}; | \Re (z)-E_0| <  \dfrac{1}{C}, | \Im(z) | < \dfrac{\epsilon_0}{C} \right\rbrace \quad \text{where $C > 0$ is a constant} .$$
Then, for $0 \leq \epsilon < \epsilon_0$, we have:
$$ \sigma(P^{ \epsilon}_{ \hbar}) \cap R_{C, \epsilon_0} = \lbrace g^{ \epsilon}_{ \hbar}(\hbar k), k \in \mathbb{Z} \rbrace \cap R_{C, \epsilon_0} + \mathcal{O}(\hbar^{ \infty} ),$$
where $g^{ \epsilon}_{ \hbar}$ is an analytic function admitting an asymptotic expansion in powers of $\hbar$ and depending smoothly on $ \epsilon$. Moreover, the first term in the asymptotic expansion of $g^{ \epsilon}_{ \hbar}$, denoted by $g^{ \epsilon}_0$, is the inverse of the action coordinate $ \tilde{I}$.
\end{theoa}

\subsection{Proof of Theorem \ref{theo_cas_simple}}
To prove Theorem \ref{theo_cas_simple}, we proceed as follows.
\begin{enumerate}
\item[Step 1:] we construct a canonical transformation $ \kappa$ and complex action-angle coordinates $( \tilde{\theta}, \tilde{I})$, where :
\begin{align*}
\kappa: \Vois( \Lambda_E^{ \epsilon}, (T^* \mathbb{S}^1)_{ \mathbb{C}}) & \longrightarrow \Vois( \tilde{I}=cst, (T^* \mathbb{S}^1)_{ \mathbb{C}}), \\
( \theta, I) & \longmapsto ( \tilde{\theta}, \tilde{I}),
\end{align*}
such that:
$$ p^{ \epsilon} \circ \kappa^{-1}( \tilde{\theta}, \tilde{I}) = g^{ \epsilon}( \tilde{I}).$$
\item[Step 2:] we quantize the canonical transformation $ \kappa$, by following this procedure:
\begin{itemize}
\item[1.] we conjugate, by a unitary transform, the operator $P_{ \hbar}^{ \epsilon}$ acting on $L^2( \mathbb{S}^1)$ in an operator $ \tilde{P}_{ \hbar}^{\epsilon}$ acting on some Bargmann space, therefore their spectra are equal;
\item[2.] we construct a unitary operator $\tilde{U}_0$ such that microlocally:
$$ \tilde{U}_0 \tilde{P}^{ \epsilon}_{ \hbar} = g^{ \epsilon} \left( \dfrac{\hbar}{i} \dfrac{\partial}{\partial \tilde{\theta}} \right)  \tilde{U}_0 + \mathcal{O}(\hbar);$$
then, by an iterative procedure, we construct a unitary operator $\tilde{U}$ such that microlocally:
$$ \tilde{U} \tilde{P}^{ \epsilon}_{ \hbar} = g^{ \epsilon}_{ \hbar} \left( \dfrac{\hbar}{i} \dfrac{\partial}{\partial \tilde{\theta}} \right) \tilde{U} + \mathcal{O}(\hbar^{ \infty}).$$
\end{itemize}
\item[Step 3:] we determine the spectrum of the operator $P^{ \epsilon}_{ \hbar}$ by using two Grushin problems, one for the operator $ \tilde{P}_{ \hbar}^{ \epsilon}$ and one for the operator $ g_{ \hbar}^{ \epsilon} \left( \dfrac{\hbar}{i} \dfrac{\partial}{\partial \tilde{\theta}} \right)$ obtained in Step 2.
\end{enumerate}

\subsubsection{Construction of the canonical transformation $ \kappa$}
This construction is analogous to what is done in \cite{MR2036816}. \\
We consider:
$$ \Lambda_E^{ \epsilon} = \lbrace ( \theta, I) \in \tub(\mathbb{S}^1 \times \mathbb{R}); p^{ \epsilon}( \theta, I) = E \rbrace .$$
Notice that the function $p^{ \epsilon}-E$ is holomorphic and that:
$$ \dfrac{\partial p^{ \epsilon}}{\partial I}( \theta, I) = \dfrac{d f^{ \epsilon}}{d I}(I) + i \epsilon \dfrac{\partial q^{ \epsilon}}{\partial I}( \theta, I) \neq 0 ,$$
for $ \epsilon$ sufficiently small because $ (f^{ \epsilon})' \neq 0$ and $ \partial_I q^{ \epsilon}$ is bounded on $ \Lambda_E^{ \epsilon}$. Therefore by applying the holomorphic implicit function theorem, we obtain that $ \Lambda_E^{ \epsilon}$ can be written as:
$$ \Lambda_E^{ \epsilon} = \lbrace ( \theta, I) \in \tub(\mathbb{S}^1 \times \mathbb{R}); I= l^{ \epsilon}( \theta,E) \rbrace,$$
where $l^{ \epsilon}$ is a holomorphic function depending smoothly on $ \epsilon$. \\
We can now define an action coordinate $\tilde{I}$ by integrating the $1$-form $Id \theta$. Since $ \Lambda_E^{ \epsilon}$ is homotopy equivalent to $ \mathbb{S}^1$, then there exists a unique loop $ \gamma_E$ in $ \Lambda_E^{ \epsilon}$ whose homotopy class generates $ \pi_1( \Lambda_E^{ \epsilon})$ (up to orientation), and we define the coordinate $ \tilde{I}$ by:
$$ \tilde{I}(E) = \dfrac{1}{2\pi} \int_{ \gamma_E} I d \theta. $$
We can choose the loop $ \gamma_E$ defined by the following parametrization, for $t \in [0, 2 \pi[$:
$$
\left\lbrace
\begin{split}
\theta(t) & = t; \\
I(t) & = l^{ \epsilon}( \theta(t), E).
\end{split}
\right.$$
Therefore, we can rewrite $ \tilde{I}(E)$ as:
$$ \tilde{I}(E) = \dfrac{1}{2 \pi} \int_0^{ 2 \pi} l^{ \epsilon}( \theta(t), E) dt.$$
Since the $1$-form $ \left. I d \theta \right|_{ \Lambda_E^{ \epsilon}}$ is closed, by applying Stokes formula we obtain that $ \tilde{I}(E)$ depends only on the homotopy class of the loop $ \gamma_E$ in $ \Lambda_E^{ \epsilon}$.\\
Moreover, notice that (since $f^{ \epsilon}$ is a local diffeomorphism):
\begin{align*}
\dfrac{d \tilde{I}}{dE}(E) & = \dfrac{1}{2 \pi} \int_0^{2 \pi} \dfrac{\partial l^{ \epsilon}}{\partial E}( \theta(t) , E) dt, \\
& = \dfrac{1}{2 \pi} \int_0^{2 \pi} \left( \dfrac{d(f^{ \epsilon})^{-1}}{dE}(E) + \mathcal{O}( \epsilon) \right) dt, \\
& = \dfrac{d(f^{ \epsilon})^{-1}}{dE}(E) + \mathcal{O}( \epsilon) \neq 0.
\end{align*}
Therefore by using the holomorphic inverse function theorem, we see that the map $ E \longmapsto \tilde{I}(E)$ is a local diffeomorphism. \\

\noindent We are now able to construct the canonical transformation $ \kappa$.

\begin{prop} \label{prop_kappa}
There exists a canonical transformation:
\begin{align*}
\kappa: \Vois( \Lambda_E^{ \epsilon}, (T^* \mathbb{S}^1)_{ \mathbb{C}}) & \longrightarrow \Vois( \tilde{I}=cst , (T^* \mathbb{S}^1)_{ \mathbb{C}}), \\
( \theta, I) & \longmapsto ( \tilde{\theta}, \tilde{I}),
\end{align*}
such that:
$$ p^{ \epsilon} \circ \kappa^{-1}( \tilde{\theta}, \tilde{I}) = g^{ \epsilon}( \tilde{I}).$$
Hence, the function $g^{ \epsilon}$ is the inverse of the action integral $ \tilde{I}$.
\end{prop}

\begin{proof}
Let $ \delta_0$ be a positive real number and let $ \mathbb{S}^1 + i ]- \delta_0, \delta_0[$ be the projection (on the first coordinate) of the tubular neighbourhood of $\mathbb{S}^1 \times \mathbb{R}$ used in the definition of the level set $ \Lambda_E^{ \epsilon}$. Let $ \pi : \mathbb{R} + i ]- \delta_0, \delta_0[ \longrightarrow \mathbb{S}^1 + i ]- \delta_0, \delta_0[$ be the projection. We denote by $ \hat{\theta}$ some complex number such that $\theta = \pi( \hat{\theta})$ with $\theta \in \mathbb{S}^1 + i ]- \delta_0, \delta_0[$. \\
We are going to prove that there exists a holomorphic function $h( \theta, \tilde{I})$ such that locally we have:
\begin{align*}
\kappa: \Vois( \Lambda_E^{ \epsilon}, (T^* \mathbb{S}^1)_{ \mathbb{C}}) & \longrightarrow  \Vois(\tilde{I}=cst, (T^* \mathbb{S}^1)_{ \mathbb{C}}), \\
\left( \theta, \dfrac{\partial h}{\partial \theta} \right) & \longmapsto \left( \dfrac{\partial h}{\partial \tilde{I}}, \tilde{I} \right).
\end{align*}
Recall that $E$ is a fixed complex number; we consider the $1$-form:
$$ \omega = \left. I d \theta \right|_{ \Lambda_E^{ \epsilon}} = l^{ \epsilon}( \theta, E) d \theta .$$
$l^{ \epsilon}$ is a holomorphic function, so we have: $d \omega = 0$ on $ \Lambda_E^{ \epsilon}$. \\
Thus, since $ \Lambda_E^{ \epsilon}$ is homotopic to $ \mathbb{S}^1$, there exists a function $h( \theta, E)$ defined on $ \Lambda_E^{ \epsilon}$ such that $dh = \omega$ if and only if:
$$ \dfrac{1}{2 \pi} \int_{ \gamma_E} \omega = 0,$$
where $ \gamma_E$ is the loop previously defined. \\
Therefore, there exists a function $h( \theta, E)$ on $\Lambda_E^{ \epsilon}$ such that:
$$ dh = \omega - \dfrac{1}{2 \pi} \int_{ \gamma_E} \omega = \omega - \dfrac{1}{2 \pi} \int_{  \gamma_E} I d \theta = \omega - \tilde{I}(E) .$$
Moreover, there also exists a function $ \hat{h}( \hat{\theta}, E)$ on $( \mathbb{R} + i ]- \delta_0, \delta_0[) \times \mathbb{C}$ such that $d \hat{h} = \hat{ \omega} = \pi^* \omega$. We can choose:
$$ \hat{h}( \hat{\theta},E) = h( \pi( \hat{\theta}),E) + \hat{\theta} \tilde{I}(E).$$
Then:
$$ d \hat{h} = \pi^*dh + \tilde{I} = \hat{\omega} .$$
Since $E \longmapsto \tilde{I}(E)$ is a local diffeomorphism, then we define a function $ \check{h}( \hat{\theta}, \tilde{I})$ by:
$$ \check{h}( \hat{\theta}, \tilde{I}) = \hat{h}( \hat{\theta}, E( \tilde{I})) = h( \pi( \hat{\theta}), E( \tilde{I})) + \hat{\theta} \tilde{I}, $$
where $E( \tilde{I})$ is the inverse function of $ \tilde{I}(E)$. \\
By definition of $ \check{h}$, there exists a function $\tilde{h}( \theta, \tilde{I})$ defined by:
$$ \tilde{h}( \theta, \tilde{I}) = \check{h}( \pi^* \theta, \tilde{I}).$$
Let:
$$ \kappa( \theta, I) = \left( \dfrac{\partial \tilde{h}}{\partial \tilde{I}}( \theta, \tilde{I}), \tilde{I} \right) .$$
This function is well-defined because it does not depend on the choice of the class representative of $ \theta$. Besides, for $( \theta, I) \in \Lambda_E^{ \epsilon}$, by construction we have:
$$ I = l^{ \epsilon}( \theta, E) = \dfrac{\partial \tilde{h}}{\partial \theta}( \theta, \tilde{I}).$$
Therefore $ \kappa$ is locally a holomorphic symplectic transformation which sends $ \Lambda_E^{ \epsilon}$ on $\lbrace \tilde{I} = cst \rbrace$ (because $ \tilde{I}$ depends only on $E$), thus:
$$ p^{ \epsilon} \circ \kappa^{-1}( \tilde{\theta}, \tilde{I}) = g^{ \epsilon}( \tilde{I}) ,$$
because $ \dfrac{\partial}{\partial \tilde{\theta}}$ the tangent vector field to $ \lbrace \tilde{I} = cst \rbrace$ is sent by $ \kappa^{-1}$ on the tangent vector field to $ \Lambda_E^{ \epsilon}$, in other words:
$$ \dfrac{\partial}{\partial \tilde{\theta}}( p^{ \epsilon} \circ \kappa^{-1}( \tilde{\theta}, \tilde{I})) = 0 \quad \Rightarrow \quad p^{ \epsilon} \circ \kappa^{-1}( \tilde{\theta}, \tilde{I}) = g^{ \epsilon}( \tilde{I}).$$
Moreover, we can deduce from this equation that $g^{ \epsilon}$ is the inverse of the action integral $ \tilde{I}$ because:
$$ g^{ \epsilon}( \tilde{I}(E)) = p \circ \kappa^{-1}( \tilde{\theta}, \tilde{I}(E)) = E.$$
\end{proof}

\begin{remark}
If $ \epsilon = 0$, $ \kappa$ is the identity (of generating function $h( \theta, \tilde{I}) = \theta \tilde{I}$).
\end{remark}

\subsubsection{Quantization of the canonical transformation $ \kappa$}
We want to construct an operator $U_0$ associated with the canonical transformation $ \kappa$. In this case, we can not apply Egorov's theorem, therefore we are going to write the canonical transformation $ \kappa$ as a composition of canonical transformations that will be easier to quantize. Before doing so, recall that one can quantize a canonical transformation if it comes from some FBI transform (see for example \cite[Chapter 13]{MR2952218}).
\\

\noindent \texttt{Notation:} Let $ \Phi$ be a strictly plurisubharmonic $ \mathbb{R}$-valued quadratic form on $ \mathbb{C}$. We introduce the following notation:
\begin{itemize}
\item $L(dz)$ is the Lebesgue measure $\dfrac{i}{2} dz \wedge d \overline{z}$;
\item $L^2( \mathbb{C}, \Phi) = L^2( \mathbb{C}, e^{-2 \Phi/\hbar} L( d z))$ is the set of measurable functions $f$ such that:
$$ \int_{ \mathbb{C}} |f(z)|^2 e^{-2 \Phi(z)/\hbar} L(dz) < + \infty ;$$
\item $L^2( \mathbb{C}, \Phi, m) = L^2( \mathbb{C}, m^2 e^{-2 \Phi/\hbar} L( d z))$ is the set of measurable functions $f$ such that:
$$ \int_{ \mathbb{C}} |f(z)|^2 m(z)^2 e^{-2 \Phi(z)/\hbar} L(dz) < + \infty ;$$
where $m$ is a function (from now one, $m$ will denote the order function associated with the operator $P_{ \hbar}^{ \epsilon}$ in Assumption (A));
\item $H( \mathbb{C}, \Phi) = Hol( \mathbb{C}) \cap L^2( \mathbb{C}, \Phi) $ is the set of holomorphic functions in $L^2( \mathbb{C}, \Phi)$;
\item $H( \mathbb{C}, \Phi, m) = Hol( \mathbb{C}) \cap L^2( \mathbb{C}, \Phi, m) $ is the set of holomorphic functions in $L^2( \mathbb{C}, \Phi, m)$.
\end{itemize}

\begin{remark}
Since the order function $m$ is such that $m \geq 1$, we have: 
$$ H( \mathbb{C}, \Phi, m) \subset H( \mathbb{C}, \Phi).$$
\end{remark}

\noindent Recall the definition of the FBI (Fourier-Bros-Iagoniltzer) transform in dimension one (see for example \cite[Chapter 13]{MR2952218}).

\begin{defi}[FBI transform and its canonical transformation] \label{defi_FBI_trnasform}
Let $ \phi(z,x)$ be a holomorphic quadratic function on $ \mathbb{C} \times \mathbb{C}$ such that:
\begin{enumerate}
\item[1.] $ \Im \left( \dfrac{\partial^2 \phi}{\partial x^2} \right)$ is a positive real number;
\item[2.] $ \dfrac{\partial^2 \phi}{\partial x \partial z} \neq 0$.
\end{enumerate}
The FBI transform associated with the function $ \phi$ is the operator $T_{ \phi}$ defined on $ \mathcal{S}( \mathbb{R})$ by:
$$ T_{ \phi} u(z) = \dfrac{c_{ \phi}}{\hbar^{3/4}} \int_{ \mathbb{R}} e^{(i/\hbar) \phi(z,x)} u(x) dx,$$
where:
$$
c_{ \phi} = \dfrac{1}{2^{1/2} \pi^{3/4}} \dfrac{| \det \partial_x \partial_z \phi|}{(\det \Im (\partial^2_x \phi ))^{1/4}}.
$$
We define a canonical transformation associated with $ T_{\phi}$ by:
\begin{align*}
\kappa_{ \phi}: \mathbb{C} \times \mathbb{C} & \longrightarrow \mathbb{C} \times \mathbb{C}, \\
(x, - \partial_x \phi(z,x)) & \longmapsto (z, \partial_z \phi(z,x)).
\end{align*}
\end{defi}

\noindent We have the following property on FBI transform (see for example \cite[p.309]{MR2952218}).

\begin{prop} \label{prop_Tphi}
Let, for $z \in \mathbb{C}$:
$$ \Phi(z) = \sup_{x \in \mathbb{R}} \left( - \Im (\phi(z,x)) \right) .$$
Then $T_{ \phi}: L^2( \mathbb{R}) \longrightarrow H( \mathbb{C}, \Phi)$ is a unitary transformation. \\
Moreover, if $T_{ \phi}^* : L^2( \mathbb{C}, \Phi) \longrightarrow L^2( \mathbb{R})$ is the adjoint of $T_{ \phi}$, then:
$$ T_{ \phi}^* v(x) = c_{ \phi} h^{-3/4} \int_{ \mathbb{C}} e^{(i/ \hbar) \overline{\phi(z,x)}} e^{-2 \Phi(z) / \hbar} v(z) L(dz).$$
And we have:
\begin{enumerate}
\item[1.] $T_{ \phi} T_{ \phi}^* = 1$ on $H( \mathbb{C}, \Phi)$;
\item[2.] $T_{ \phi}^* T_{ \phi} = 1$ on $L^2( \mathbb{R})$.
\end{enumerate}
\end{prop}

\begin{remark}
The canonical transformation $ \kappa_{ \phi}$ sends $ \mathbb{R}^2$ on the $IR$-manifold ($I$-Lagrangian and $R$-symplectic) $ \Lambda_{ \Phi} = \left\lbrace \left( z, \dfrac{2}{i} \dfrac{\partial \Phi}{\partial z}(z) \right); z \in \mathbb{C} \right\rbrace$ where $ \Phi$ is a strictly plurisubharmonic $ \mathbb{R}$-valued quadratic form associated with $ \phi$ in the sense of Proposition \ref{prop_Tphi}.
\end{remark}

\noindent First, we have the following results (see for example \cite[p.139-142]{SJ} or \cite{MR2003421}).

\begin{prop}\label{prop_H_phi_0}
Let $P^{ \epsilon}_{ \hbar}$ be a pseudo-differential operator acting on $L^2( \mathbb{R})$ and satisfying the hypothesis (A) to (D). Let $ \Phi_0$ be a strictly plurisubharmonic $ \mathbb{R}$-valued quadratic form on $ \mathbb{C}$ (we can associate with $\Phi_0$ a holomorphic quadratic function $ \phi_0$ in the sense of Proposition \ref{prop_Tphi}).
Let  $\tilde{P}^{ \epsilon}_{ \hbar} = T_{ \phi_0} \circ P^{ \epsilon}_{ \hbar} \circ T_{ \phi_0}^*$. Then:
\begin{enumerate}
\item[1.] $ \tilde{P}^{ \epsilon}_{ \hbar} : H( \mathbb{C}, \Phi_0, \tilde{m}) \longrightarrow H( \mathbb{C}, \Phi_0)$ is uniformly bounded in $\hbar$ and $ \epsilon$ (for $ \hbar < 1$ and $ \epsilon < \epsilon_0$ where $ \epsilon_0 $ is a fixed positive real number), where $ \tilde{m} = m \circ \kappa_{ \phi_0}^{-1}$ is an order function on $ \Lambda_{ \Phi_0} = \left\lbrace (y, \eta) \in \mathbb{C}^2; \eta = \dfrac{2}{i} \dfrac{\partial \Phi_0}{\partial y}(y) \right\rbrace$ (recall that $m$ is the order function associated with the operator $P_{ \hbar}^{ \epsilon}$ in Hypothesis (A));
\item[2.] $ \tilde{P}^{ \epsilon}_{ \hbar}$ is given by the contour integral:
$$ \tilde{P}^{ \epsilon}_{ \hbar} u(x) = \dfrac{1}{2 \pi \hbar} \int \! \! \! \int_{\Gamma(x)} e^{(i/\hbar) (x-y)\eta} \tilde{p}^{ \epsilon}_{ \hbar} \left( \dfrac{x+y}{2}, \eta \right) u(y) dy d\eta, $$
where $ \Gamma(x) = \left\lbrace (y, \eta) \in \mathbb{C}^2; \eta = \dfrac{2}{i} \dfrac{\partial \Phi_0}{\partial x} \left( \dfrac{x+y}{2} \right) \right\rbrace$ and where the symbol $ \tilde{p}^{ \epsilon}_{ \hbar}$ is given by $ \tilde{p}^{ \epsilon}_{ \hbar} = p^{ \epsilon}_{ \hbar} \circ \kappa_{ \phi_0}^{-1}$.
\end{enumerate} 
\end{prop}

\noindent Since $ \tilde{p}^{ \epsilon}_{ \hbar}$ is a holomorphic function and is bounded by the order function $\tilde{m}$ in a tubular neighbourhood of $ \Lambda_{ \Phi_0}$, we can perform a contour deformation of $ \Gamma(x)$ and consider other weight functions as follows.

\begin{prop} \label{prop_H_Phi}
With the notation of Proposition \ref{prop_H_phi_0}, let $ \Phi \in \mathcal{C}^{1,1}( \mathbb{C}, \mathbb{R})$ (the space of $ \mathcal{C}^1$ functions with Lipschitz gradient) be a function close to $ \Phi_0$ in the following sense:
\begin{enumerate}
\item[1.] $ \Phi - \Phi_0$ is bounded;
\item[2.] there exists a constant $C > 0$ such that: $ \sup \left| \dfrac{\partial \Phi}{\partial x} - \dfrac{\partial \Phi_0}{\partial x} \right| < \dfrac{1}{2C}$, where $C$ is large enough, so that:
$$ \Gamma_C(x) = \left\lbrace (y, \eta) \in \mathbb{C}^2; \eta = \dfrac{2}{i} \dfrac{\partial \Phi_0}{\partial x} \left( \dfrac{x+y}{2} \right) + \dfrac{i}{C} \dfrac{\overline{x-y}}{\langle x-y \rangle} \right\rbrace \subset \Lambda_{ \Phi} .$$
\end{enumerate}
Then $ \tilde{P}^{ \epsilon}_{ \hbar}: H( \mathbb{C}, \Phi, \tilde{m}) \longrightarrow H( \mathbb{C}, \Phi)$ is uniformly bounded in $\hbar$ and $ \epsilon$ (for $ \hbar < 1$ and $ \epsilon < \epsilon_0$ where $ \epsilon_0 $ is a fixed positive real number).
\end{prop}

We now introduce the following strictly plurisubharmonic quadratic form :
\begin{align*}
\Phi_1: \mathbb{C} & \longrightarrow \mathbb{R} \\
x & \longmapsto \dfrac{1}{2} | \Im(x)|^2
\end{align*}
This quadratic form is associated in the sense of Proposition \ref{prop_Tphi} with the holomorphic quadratic function $ \phi_1$ defined by, for all $z, x \in \mathbb{C}$:
$$ \phi_1(z,x) = \dfrac{i}{2} (z-x)^2 .$$
The canonical transformation $ \kappa_{ \phi_1}$ is given by:
\begin{align*}
\kappa_{ \phi_1}: \mathbb{C} \times \mathbb{C} & \longrightarrow \mathbb{C} \times \mathbb{C}, \\
(x, \xi) & \longmapsto (x-i \xi,\xi).
\end{align*}
Notice that, for $(x, \xi) \in \mathbb{C}^2$, we have:
$$ \kappa_{ \phi_1}(x+2 \pi, \xi) = \kappa_{ \phi_1}(x, \xi) + (2 \pi, 0).$$
Therefore, there exists a map $ \overline{\kappa}_{ \phi_1}: ( \mathbb{S}^1 + i \mathbb{R}) \times \mathbb{C} \longrightarrow ( \mathbb{S}^1 + i \mathbb{R}) \times \mathbb{C}$ such that $ \pi \circ \kappa_{ \phi_1} = \overline{\kappa}_{ \phi_1} \circ \pi$ where $ \pi: ( \mathbb{R} + i \mathbb{R}) \times \mathbb{C} \longrightarrow ( \mathbb{S}^1 + i \mathbb{R}) \times \mathbb{C}$ is the projection.\\

\noindent We consider the following transformations:
\begin{enumerate}
\item[1.] $ \overline{\kappa}_{ \phi_1} : (T^* \mathbb{S}^1)_{ \mathbb{C}} \longrightarrow (T^* \mathbb{S}^1)_{ \mathbb{C}}$ which sends $ \mathbb{S}^1 \times \mathbb{R}$ to $ \Lambda_{ \Phi_1}$ where:
$$ \Lambda_{ \Phi_1} = \left\lbrace ( x, \xi) \in (T^* \mathbb{S}^1)_{ \mathbb{C}}, \xi = \dfrac{2}{i} \dfrac{\partial \Phi_1}{\partial x}(x) = - \Im(x) \right\rbrace; $$
\item[2.] $ \tilde{\kappa}^{-1}$ defined by:
$$ \tilde{\kappa}^{-1} = \overline{\kappa}_{ \phi_1}^{-1} \circ \kappa^{-1} \circ \overline{\kappa}_{ \phi_1}: (T^* \mathbb{S}^1)_{ \mathbb{C}} \longrightarrow (T^* \mathbb{S}^1)_{ \mathbb{C}},$$
which does not preserve $ \Lambda_{ \Phi_1}$ (because $ \kappa$ is not a real transformation) but sends it to another $IR$-manifold denoted by $ \Lambda_{ \Phi_2}$, where $\Phi_2$ is a smooth function close to $ \Phi_1$.
\end{enumerate} 
To summarize, we consider the following commutative diagram on the phase spaces:
\begin{displaymath}
    \xymatrix{
        \underset{\color{red}{( \theta, I)}}{\mathbb{S}^1 \times \mathbb{R} \subset (T^* \mathbb{S}^1)_{ \mathbb{C}}} \ar[r]^{ \kappa} \ar[d]_{ \overline{\kappa}_{ \phi_1}} & \underset{\color{red}{( \tilde{\theta}, \tilde{I})}}{(T^* \mathbb{S}^1)_{ \mathbb{C}} \supset \mathbb{S}^1} \times \mathbb{R} \ar[d]^{ \overline{\kappa}_{ \phi_1}}\\
        \overset{\Lambda_{ \Phi_1} \subset (T^* \mathbb{S}^1)_{ \mathbb{C}}}{\underset{ \color{red}{(y, \eta)}}{\Lambda_{ \Phi_2} \subset (T^* \mathbb{S}^1)_{ \mathbb{C}}}}   & \underset{\color{red}{(x, \xi)}}{(T^* \mathbb{S}^1)_{ \mathbb{C}} \supset \Lambda_{ \Phi_1}}\ar[l]_{\tilde{\kappa}^{-1}} }
\end{displaymath}

We want to quantize the previous transformations.
First, we show how to construct a unitary operator associated with the transformation $ \tilde{\kappa}$, following \cite{MR2003421} (note that their case is the two dimensional one). For the sake of completeness, we recall the one dimension theory. We consider:
\begin{align*}
\tilde{\kappa}^{-1} = \overline{\kappa}_{ \phi_1} \circ \kappa^{-1} \circ \overline{\kappa}_{ \phi_1}^{-1}: (T^* \mathbb{S}^1)_{ \mathbb{C}} & \longrightarrow (T^* \mathbb{S}^1)_{ \mathbb{C}}, \\
(x, \xi) & \longmapsto (y, \eta).
\end{align*}
\noindent First, we can show that there exists a smooth function $ \Phi_2$ such that the transformation $ \tilde{\kappa}^{-1}$ sends the $IR$-manifold $ \Lambda_{ \Phi_1}$ to $ \Lambda_{ \Phi_2}$.

\begin{prop} \label{prop_tilde(kappa)_envoie_Phi2_sur_Phi1}
There exists a smooth function $ \Phi_2$ such that:
\begin{enumerate}
\item[1.] $ \Phi_2$ is uniformly strictly plurisubharmonic;
\item[2.] $ \Phi_2$ is close to $\Phi_1$ in the sense of Proposition \ref{prop_H_Phi};
\item[3.] $ \tilde{\kappa}^{-1}( \Lambda_{ \Phi_1}) = \Lambda_{ \Phi_2} = \left\lbrace (y, \eta) \in (T^* \mathbb{S}^1)_{ \mathbb{C}}; \eta = \dfrac{2}{i} \dfrac{\partial \Phi_2}{\partial y}(y) \right\rbrace$.
\end{enumerate}
\end{prop}

\begin{proof}
Since $ \tilde{\kappa}^{-1}$ is the composition of three holomorphic symplectic transformations, then $ \tilde{\kappa}^{-1}( \Lambda_{ \Phi_1})$ is an $IR$-manifold. Thus, by using the fact that $ \kappa$ is close to the identity map when $ \epsilon$ is small, we can show that $ \tilde{\kappa}^{-1}( \Lambda_{ \Phi_1})$ can be written as $ \Lambda_{ \Phi_2}$ with $ \Phi_2$ some smooth function. Besides, since $ \epsilon$ is small, $ \tilde{\kappa}^{-1}$ is close to the identity, therefore $ \Lambda_{ \Phi_2}$ is close to $ \Lambda_{ \Phi_1}$, so $ \partial_x \Phi_2(x) $ is close to $ \partial_x \Phi_1(x) $. Besides, since $ \Phi_1$ is a uniformly strictly plurisubharmonic function, $ \tilde{\kappa}$ is holomorphic and $ \partial_x \Phi_2(x)$ is close to $ \partial_x \Phi_1(x)$, then $ \Phi_2$ is also a uniformly plurisubharmonic function.
\end{proof}

\noindent Let $ graph( \tilde{\kappa}) = \lbrace ( x, \xi; y,  \eta) \in \Lambda_{ \Phi_1} \times \Lambda_{ \Phi_2}; (x, \xi) = \tilde{\kappa}(y, \eta) \rbrace$. Following \cite{MR2003421}, we can construct a function $\psi(x,y)$, defined in a neighbourhood of each point of $ \pi_{x, y}( graph( \tilde{\kappa}))$ (the projection of the set $graph( \tilde{\kappa})$), such that:
\begin{enumerate}
\item[1.] $ \partial_{ \overline{x}} \psi(x,y)$ and $ \partial_y \psi(x,y)$ vanish to infinite order on  $ \pi_{(x,y)}( graph( \tilde{\kappa}))$;
\item[2.] $ \partial_x \psi(x,y) = \dfrac{2}{i} \dfrac{\partial \Phi_1}{\partial x}(x)$ and $ \partial_{ \overline{y}} \psi(x,y) = \dfrac{2}{i} \dfrac{\partial \Phi_2}{\partial \overline{y}}(y)$,
$ \forall (x,y) \in \pi_{(x,y)}( graph( \tilde{\kappa}))$;
\item[3.] $ \Phi_1(x) + \Phi_2(y) + \Im (\psi(x,y)) \sim dist((x,y), \pi_{(x,y)}( graph( \tilde{\kappa}))^2$.
\end{enumerate}
\begin{remark}
We can consider the set $ \pi_{x, y}( graph( \tilde{\kappa}))$ because $ \Lambda_{ \Phi_1}$ and $ \Lambda_{ \Phi_2}$ are parametrized by $x$ and $y$ respectively, so $ \Lambda_{ \Phi_1} \times \Lambda_{ \Phi_2}$ too. Therefore $ \pi_{x, y}( graph( \tilde{\kappa}))$ is a regular submanifold of $ \Lambda_{ \Phi_1} \times \Lambda_{ \Phi_2}$.
\end{remark}

\noindent According to Conditions 1. and 2. we have:
$$ d \psi = \dfrac{2}{i} \dfrac{\partial \Phi_1}{\partial x}(x) dx + \dfrac{2}{i} \dfrac{\partial \Phi_2}{\partial \overline{y}}(y) d \overline{y} \quad \text{on $\pi_{(x,y)}( graph( \tilde{\kappa}))$}.$$
If we restrict $ \psi$ to $ \pi_{(x,y)}( graph( \tilde{\kappa}))$ and identify it with a function on $ graph( \tilde{\kappa})$, we obtain:
$$ d( \left. \psi \right|_{ graph( \tilde{\kappa})} )= \xi dx - \overline{ \eta} d \overline{y} \quad \text{ for $(x, \xi; y, \eta) \in graph( \tilde{\kappa})$}.$$
We want to study the analytic continuation of the function $ \psi$ along a loop $ \gamma$ in $ graph( \tilde{\kappa})$. \\
First, notice that:
\begin{align*}
\left. \Im ( \xi dx ) \right|_{ \Lambda_{ \Phi_1}} & = \Im \left( \dfrac{2}{i} \dfrac{\partial \Phi_1}{\partial x} dx \right), \\
& = \dfrac{1}{2i} \left( \dfrac{2}{i} \dfrac{\partial \Phi_1}{\partial x} dx - \overline{\dfrac{2}{i} \dfrac{\partial \Phi_1}{\partial x} dx } \right), \\
& = - \left( \dfrac{\partial \Phi_1}{ \partial x} dx + \dfrac{\partial \Phi_1}{\partial \overline{x}}  d \overline{x}\right), \\
& = - d \Phi_1.
\end{align*}
So the form $\left. \Im ( \xi dx ) \right|_{ \Lambda_{ \Phi_1}}$ is exact. Similarly $ \left. \Im ( \overline{\eta} d \overline{y} ) \right|_{ \Lambda_{ \Phi_2}}$ is exact. \\
Let $ \hat{\gamma} = \lbrace ( \tilde{\kappa}( \rho), \rho); \rho \in \gamma \rbrace$ where $\gamma$ is any loop in the domain of $ \tilde{\kappa}$ restricted to $ \Lambda_{ \Phi_2}$. We have:
\begin{align*}
\int_{ \hat{\gamma}} d \psi & = \int_{( \tilde{ \kappa} \circ \gamma, \gamma)} \left( \xi dx - \overline{ \eta} d \overline{y} \right), \\
& = \int_{ \tilde{\kappa} \circ \gamma} \xi dx - \int_{ \gamma} \overline{\eta} d \overline{y}, \\
& = \int_{ \tilde{\kappa} \circ \gamma} \left( \Re( \xi dx) + i \Im( \xi dx) \right) - \int_{ \gamma}  \left( \Re( \overline{\eta} d \overline{y}) + i \Im( \overline{\eta} d \overline{y}) \right),
\\
& = \int_{ \tilde{\kappa} \circ \gamma} \Re( \xi dx) - \int_{ \gamma} \Re( \eta dy), \\
& := - J(\gamma).
\end{align*}
Therefore along a loop, $ \psi$ changes by a real constant since it is the difference of two real actions. We call this difference the Floquet index, this number depends only on $ \tilde{\kappa}$. \\

\noindent \texttt{Notation:}
\begin{itemize}
\item $L^2_J( \mathbb{S}^1)$ is the space of Floquet periodic measurable functions $f$ such that:
$$ \dfrac{1}{2 \pi} \int_0^{2 \pi} |f(x)|^2 dx < + \infty ,$$
such a function $f$ satisfies the following Floquet periodicity condition:
$$ f(x+2 \pi) = e^{-(i/ \hbar) J} f(x).$$
\item $L^2_J( \mathbb{S}^1 + i \mathbb{R}, \Phi)$ is the space of multi-valued Floquet periodic functions $f$ such that:
$$ \int_0^{ 2 \pi} \int_{ \mathbb{R}} |f(z)|^2 e^{-2 \Phi(z)/ \hbar} L(dz) < + \infty ,$$
\item $H_J( \mathbb{S}^1 + i \mathbb{R}, \Phi)$ is the space of holomorphic functions in $L^2_J( \mathbb{S}^1 + i \mathbb{R}, \Phi)$.
\end{itemize}

\noindent We can now quantize the transformation $ \tilde{\kappa}$.

\begin{prop}[\cite{MR2003421}] \label{prop_quantif_A}
Let $A$ be the operator defined by:
$$ A u( x) = \dfrac{1}{\hbar} \int_{ \mathbb{C}} e^{(i/\hbar)\psi( x,y)} a( x,y) \chi(x, y) u( y) e^{-(2/\hbar) \Phi_2( y)} L( d y),$$
where $a( x, y)$ is a symbol satisfying:
\begin{enumerate}
\item[1.] $a(x,y) \sim \sum a_j( x,y) \hbar^j$ in $ \mathcal{C}^{ \infty}( \Vois( \pi_{( x,y)}( graph( \tilde{\kappa}))))$;
\item[2.] $a_j \in \mathcal{C}^{ \infty}$;
\item[3.] $ \partial_{ \overline{x}} a_j = \mathcal{O}( (dist(( x,y), \pi_{(x,y)} graph( \tilde{\kappa})))^{ \infty} + \hbar^{ \infty})$;
\item[4.] $ \partial_{ y} a_j = \mathcal{O}( (dist( (x,y), \pi_{( x,y)} graph( \tilde{\kappa})))^{ \infty} + \hbar^{ \infty})$;
\item[5.] $a$ elliptic, \textit{i.e.} $a_0$ does not vanish;
\end{enumerate} 
and where $ \chi$ is a cut-off equal to $1$ in a neighbourhood of $ \pi_{( x,y)}( graph( \tilde{\kappa}))$. \\
Let $U \subset \Lambda_{ \Phi_2}$ and let $V \subset \Lambda_{ \Phi_1}$ such that $ \tilde{\kappa}(U) = V$. Then:
\begin{enumerate}
\item[1.] $A = L^2( \pi(U), e^{-2 \Phi_2/\hbar} L( d y)) \longrightarrow L^2_{\Flo}( \pi(V), e^{-2 \Phi_1/\hbar} L( d x))$ is a bounded operator;
\item[2.] $ \| (\overline{\partial} \circ A) u \|_{L^2_{\Flo}} \leq \mathcal{O}( \hbar^{ \infty}) \| u \|_{L^2}$.
\end{enumerate}
\end{prop}

\begin{remark}
Let $A^*$ be the adjoint of $A$. Then, $A^*$ is associated with the transformation $ \tilde{\kappa}^{-1}$ and we can choose the symbol $a$ such that, up to $ \mathcal{O}(\hbar^{ \infty})$ (with the notations of Proposition \ref{prop_quantif_A}):
\begin{itemize}
\item $A^* A$ is the orthogonal projector $L^2( \pi(U), e^{-2 \Phi_2/\hbar} L( d y)) \longrightarrow H( \pi(U), \Phi_2)$;
\item $AA^*$ is the orthogonal projector $L^2_{\Flo}( \pi(V), e^{-2 \Phi_1/\hbar} L( d x)) \longrightarrow H_{\Flo}( \pi(V), \Phi_1)$.
\end{itemize}
\end{remark}

\noindent Therefore, we obtained a unitary operator $A$ microlocally defined on the $L^2( \Phi)$-spaces associated with the transformation $ \tilde{\kappa}$, which sends the set of holomorphic functions on itself up to $ \mathcal{O}( \hbar^{ \infty})$. We also have an Egorov theorem in this case, as follows.

\begin{prop}[\cite{MR2003421}]
With the notation of Proposition \ref{prop_quantif_A}, there exists an operator $ \hat{P}^{ \epsilon}_{ \hbar}$ depending smoothly on $ \epsilon$ defined by:
$$ \hat{P}^{ \epsilon}_{ \hbar} (x, \hbar D_x) u(x) = \dfrac{1}{2 \pi \hbar} \int \! \! \! \int_{ \Gamma(x)} e^{(i/\hbar)(x-y) \eta} ( \chi \hat{p}^{ \epsilon}_{ \hbar}) \left( \dfrac{x+y}{2}, \eta \right) u(y) dy d \eta,$$
where $ \Gamma(x) = \left\lbrace (y, \eta) \in (T^* \mathbb{S}^1)_{ \mathbb{C}}; \eta = \dfrac{2}{i} \dfrac{\partial \Phi_1}{\partial x} \left( \dfrac{x+y}{2} \right) \right\rbrace$ and where $ \chi$ is a suitable cut-off, such that:
\begin{enumerate}
\item[1.] the principal symbol $ \hat{p}^{ \epsilon}$ of $ \hat{P}^{ \epsilon}_{ \hbar}$ satisfies the equation $ \hat{p}^{ \epsilon}= \tilde{p}^{ \epsilon} \circ \tilde{ \kappa}^{-1}$;
\item[2.] $ \hat{P}^{ \epsilon}_{ \hbar} A = A \tilde{P}^{ \epsilon}_{ \hbar}$ and  $A^* \hat{P}^{ \epsilon}_{ \hbar} = \tilde{P}^{ \epsilon}_{ \hbar} A^*$ up to $ \mathcal{O}(\hbar^{ \infty})$ in the sense that:
\begin{align*}
\| (\hat{P}^{ \epsilon}_{ \hbar} A - A \tilde{P}^{ \epsilon}_{\hbar}) u \|_{L^2_{\Flo}( \tilde{V}, \Phi_1)} & \leq \mathcal{O}(\hbar^{ \infty}) \| u \|_{H(U, \Phi)}, \\
\| ( A^* \hat{P}^{ \epsilon}_{ \hbar} - \tilde{P}^{ \epsilon}_{ \hbar} A^*)u \|_{L^2( \tilde{U}, \Phi)} & \leq \mathcal{O}( \hbar^{ \infty}) \| u \|_{H_{ \Flo}(V, \Phi_1)},
\end{align*}
where $ \tilde{V}$ is a compact subset of $\pi(V)$ and $ \tilde{U}$ is a compact subset of $ \pi(U)$.
\end{enumerate}
\end{prop}

We previously defined an operator $T_{ \phi_1}: L^2( \mathbb{R}) \longrightarrow H( \mathbb{C}, \Phi_1)$ associated with the canonical transformation $T_{ \phi_1}$. We now want to construct an operator acting on the Floquet spaces associated with the canonical transformation $ \overline{\kappa}_{\phi_1}$, thus we are looking for an operator $B$ such that:
$$ B : H_{\Flo}( \mathbb{S}^1 + i \mathbb{R}, \Phi_1) \longrightarrow L^2_{\Flo}( \mathbb{S}^1) .$$

\noindent \texttt{Notation:} We denote by $k$ the kernel of the FBI transform $T_{ \phi_1} : L^2( \mathbb{R}) \longrightarrow H( \mathbb{C}, \Phi_1)$ associated with $ \phi_1$, \textit{i.e.}:
$$ T_{ \phi_1} u(z) = c_{ \phi_1} \hbar^{-3/4} \int_{ \mathbb{R}} e^{-(1/2 \hbar)(z - x)^2} u(x) dx = \int_{ \mathbb{R}} k(z-x; \hbar) u(x) d x ,$$
with $c_{ \phi_1} \geq 0$ the constant given by Definition \ref{defi_FBI_trnasform}. \\
The complex adjoint $T_{ \phi_1}^* : L^2( \mathbb{C}, \Phi_1) \longrightarrow L^2( \mathbb{R})$ can be rewritten as:
\begin{align*}
 T_{ \phi_1}^* v(x) & = c_{ \phi_1} \hbar^{-3/4} \int_{ \mathbb{C}} e^{-(1/2 \hbar) ( \overline{z} - \overline{x})2} e^{-2 \Phi_1(z)/ \hbar} v(z) L( dz), \\
& = \int_{ \mathbb{C}} \overline{k(z- x; \hbar)} e^{-2 \Phi_1(z)/ \hbar} v(z) L(dz).
\end{align*}
\\

\noindent We identify the functions in $L^2_{\Flo}( \mathbb{S}^1)$ with the Floquet periodic locally square integrable functions on $ \mathbb{R}$ and similarly for the functions in $H_{\Flo}( \mathbb{S}^1 + i \mathbb{R}, \Phi_1)$.

\begin{prop}[\cite{MR2003421}] $ $ 
\begin{enumerate}
\item[1.] $T_{ \phi_1}$ induces an operator $B^*: L^2_{\Flo}( \mathbb{S}^1) \longrightarrow H_{\Flo}( \mathbb{S}^1 + i \mathbb{R}, \Phi_1)$ given by:
$$ B^* u(z) = \int_{ \mathbb{R}} k(z - x; \hbar) u( x) dx = \int_{ \mathcal{E}} \sum_{ \nu \in 2 \pi \mathbb{Z}} k(z -x + \nu;\hbar) e^{(i/\hbar) J \nu} u(x) dx ,$$
where $ \mathcal{E} \subset \mathbb{R}$ is a fundamental domain for $2 \pi \mathbb{Z}$.
\item[2.] The complex adjoint of $B^*$ is defined by:
\begin{align*}
B v(x) & = \int_{\mathcal{E} + i \mathbb{R}} \overline{\sum_{ \nu \in 2 \pi \mathbb{Z}} k(z -x + \nu;\hbar) e^{(i/ \hbar) J \nu}} e^{-2 \Phi_1(z)/ \hbar} v(z) L(dz), \\
& = \int_{ \mathbb{C}} \overline{k(z- x;\hbar)} e^{-2 \Phi_1(z)/ \hbar}  v(z) L(dz).
\end{align*}
Therefore $B$ coincides with the inverse of the FBI transform $T_{ \phi_1}^*$.
\end{enumerate} 
\end{prop}

\noindent Since the FBI transform $T_{ \phi_1}$ is a unitary operator according to Proposition \ref{prop_Tphi}, then we can deduce that $B^*$ is also a unitary operator.

\begin{prop}[\cite{MR2003421}] \label{prop_B_unitaire} $ $
\begin{enumerate}
\item[1.] $B B^* = 1$ on $L^2_{\Flo}( \mathbb{S}^1)$;
\item[2.] $B^* B=1$ on $H_{\Flo}( \mathbb{S}^1 + i \mathbb{R}, \Phi_1)$.
\end{enumerate}
\end{prop}

\noindent We also have an Egorov theorem in the $L^2( \mathbb{S}^1)$-case that we deduce from Proposition \ref{prop_H_phi_0}.

\begin{prop}\label{prop_B*_Egorov}
Let $P^{ \epsilon}_{ \hbar}$ be a pseudo-differential operator acting on $L^2_{\Flo}( \mathbb{S}^1)$ and satisfying the hypothesis (A') to (D'). Let  $\tilde{P}^{ \epsilon}_{ \hbar} = B^* \circ P^{ \epsilon}_{ \hbar} \circ B$. Then:
\begin{enumerate}
\item[1.] $ \tilde{P}^{ \epsilon}_{ \hbar} : H_{\Flo}( \mathbb{S}^1 + i \mathbb{R}, \Phi_1, \tilde{m}) \longrightarrow H_{\Flo}( \mathbb{S}^1 + i \mathbb{R}, \Phi_1)$ is uniformly bounded in $\hbar$ and $ \epsilon$ (for $ \hbar < 1$ and $ \epsilon < \epsilon_0$ where $ \epsilon_0 $ is a fixed positive real number), where $ \tilde{m} = m \circ \overline{\kappa}_{ \phi_1}^{-1}$ is an order function on:
$$ \Lambda_{ \Phi_1} = \left\lbrace (y, \eta) \in (T^* \mathbb{S}^1)_{ \mathbb{C}}; \eta = \dfrac{2}{i} \dfrac{\partial \Phi_1}{\partial y}(y) \right\rbrace; $$
\item[2.] $ \tilde{P}^{ \epsilon}_{ \hbar}$ is given by the contour integral:
$$ \tilde{P}^{ \epsilon}_{ \hbar} u(x) = \dfrac{1}{2 \pi \hbar} \int \! \! \! \int_{\Gamma(x)} e^{(i/\hbar) (x-y)\eta} \tilde{p}^{ \epsilon}_{ \hbar} \left( \dfrac{x+y}{2}, \eta \right) u(y) dy d\eta, $$
where $ \Gamma(x) = \left\lbrace (y, \eta) \in (T^* \mathbb{S}^1)_{ \mathbb{C}}; \eta = \dfrac{2}{i} \dfrac{\partial \Phi_1}{\partial x} \left( \dfrac{x+y}{2} \right) \right\rbrace$ and where the symbol $ \tilde{p}^{ \epsilon}_{ \hbar}$ is given by $ \tilde{p}^{ \epsilon}_{ \hbar} = p^{ \epsilon}_{ \hbar} \circ \overline{\kappa}_{ \phi_1}^{-1}$.
\end{enumerate} 
\end{prop}

\begin{prop} \label{prop_B*_Egorov_2}
With the notation of Proposition \ref{prop_B*_Egorov}, let $ \Phi_2$ be a function of class $ \mathcal{C}^{1,1}$ close to $ \Phi_1$ in the following sense:
\begin{enumerate}
\item[1.] $ \Phi_2 - \Phi_1$ is bounded;
\item[2.] $ \sup \left| \dfrac{\partial \Phi_2}{\partial x} - \dfrac{\partial \Phi_1}{\partial x} \right|$ is sufficiently small.
\end{enumerate}
Then $ \tilde{P}^{ \epsilon}_{ \hbar}: H_J( \mathbb{S}^1 + i \mathbb{R}, \Phi_2, \tilde{m}) \longrightarrow H_J( \mathbb{S}^1 + i \mathbb{R}, \Phi_2)$ is uniformly bounded in $\hbar$ and $ \epsilon$ (for $ \hbar < 1$ and $ \epsilon < \epsilon_0$ where $ \epsilon_0 $ is a fixed positive real number).
\end{prop}

\begin{remark} Propositions \ref{prop_B*_Egorov} and \ref{prop_B*_Egorov_2} hold in the $L^2( \mathbb{S}^1)$-case.
\end{remark}

\noindent To summarize, we have the following diagram (with the notations of Propositions \ref{prop_quantif_A} and \ref{prop_B*_Egorov}):
\begin{displaymath}
    \xymatrix{
        L^2( \mathbb{S}^1) \ar[r]^{U_0} \ar[d]_{B^*} & L^2_{\Flo}( \mathbb{S}^1) \ar[d]^{B^*} \\
        \overset{H( \mathbb{S}^1 + i \mathbb{R}, \Phi_1)}{H( \pi(U), \Phi_2)} \ar[r]_{A}       & H_{\Flo}( \pi(V), \Phi_1) }
\end{displaymath}
We can apply Proposition \ref{prop_B*_Egorov_2} in the $L^2( \mathbb{S}^1)$-case and obtain an operator:
$$ \tilde{P}^{ \epsilon}_{ \hbar} = B^* P^{ \epsilon}_{ \hbar} B : H( \mathbb{S}^1 + i \mathbb{R}, \Phi_2) \longrightarrow H( \mathbb{S}^1 + i \mathbb{R}, \Phi_2) .$$
Then, if $ \tilde{U}_0=BA$ microlocally we have by composition:
$$ \tilde{U}_0 \tilde{P}^{ \epsilon}_{ \hbar}  = g^{ \epsilon} \left( \dfrac{\hbar}{i} \dfrac{\partial}{ \partial \tilde{\theta}} \right) \tilde{U}_0 + \mathcal{O}(\hbar) ,$$
where $g^{ \epsilon}$ is the function given by Proposition \ref{prop_kappa} and where $g^{ \epsilon} \left( \dfrac{\hbar}{i} \dfrac{\partial}{\partial \tilde{\theta}} \right)$ is the Weyl quantization of the symbol $g^{ \epsilon}(\tilde{I})$ on $L^2_J( \mathbb{S}^1)$. \\

\noindent We can sum up what we have done in this paragraph by the following proposition.

\begin{prop} \label{prop_tilde(U)_0}
There exists a unitary operator $\tilde{U}_0: H( \pi(U), \Phi_2) \longrightarrow L^2_{\Flo}( \mathbb{S}^1)$ such that microlocally we have:
$$ \tilde{U}_0 \tilde{P}^{ \epsilon}_{ \hbar} = g^{ \epsilon} \left( \dfrac{\hbar}{i} \dfrac{\partial}{\partial \tilde{\theta}} \right) \tilde{U}_0  + \mathcal{O}(\hbar) ,$$
where $ \pi(U)$ is a suitable neighbourhood as in Proposition \ref{prop_quantif_A} and where $g^{ \epsilon}$ is an analytic function depending smoothly on $ \epsilon$ whose inverse is the action integral $ \tilde{I}$.
\end{prop}

\noindent We can improve this proposition by using an iterative procedure.

\begin{prop} \label{prop_tilde(UPU)=S}
There exists a unitary operator $\tilde{U}: H(\pi(U), \Phi_2) \longrightarrow L^2_{\Flo}( \mathbb{S}^1)$ such that microlocally:
$$ \tilde{U} \tilde{P}^{ \epsilon}_{ \hbar} = g^{ \epsilon}_{ \hbar} \left( \dfrac{\hbar}{i} \dfrac{\partial}{\partial \tilde{\theta}} \right) \tilde{U} + \mathcal{O}(\hbar^{ \infty}),$$
where $ \pi(U)$ is a suitable neighbourhood as in Proposition \ref{prop_quantif_A} and where $g^{ \epsilon}_{ \hbar}$ is an analytic function admitting an asymptotic expansion in powers of $\hbar$, depending smoothly on $ \epsilon$ and whose first term $g^{ \epsilon}_0 := g^{ \epsilon}$ is the inverse of the action integral $ \tilde{I}$.
\end{prop}

\begin{proof}
Let $ \tilde{U}_0$ be the operator defined in Proposition \ref{prop_tilde(U)_0}, if $S_0 := g^{ \epsilon} \left( \dfrac{\hbar}{i} \dfrac{\partial}{\partial \tilde{\theta}} \right)$ then we have:
\begin{equation} \label{eq_demo_1}
 \tilde{U}_0 \tilde{P}^{ \epsilon}_{ \hbar} = S_0 \tilde{U}_0 + \mathcal{O}(\hbar) := (S_0 + \hbar R_1) \tilde{U}_0.
\end{equation}
We want to modify $ \tilde{U}_0$ to obtain our result. More precisely, we first look for a unitary operator $V$ such that:
\begin{equation} \label{eq_demo_2}
V (\tilde{U}_0 \tilde{P}^{ \epsilon}_{ \hbar})  = (S_0 + \hbar S_1) V \tilde{U}_0 +  \mathcal{O}(\hbar^2),
\end{equation}
with $S_1 = g_1^{ \epsilon} \left( \dfrac{\hbar}{i} \dfrac{\partial}{\partial \tilde{\theta}} \right)$ and where $g_1^{ \epsilon}$ is a function to determine. We have, according to Equations \eqref{eq_demo_1} and \eqref{eq_demo_2}:
\begin{align*}
V ( \tilde{U}_0 \tilde{P}^{ \epsilon}_{ \hbar} ) =V (S_0 \tilde{U}_0 + \hbar R_1 \tilde{U}_0) & = (S_0 + \hbar S_1) V \tilde{U}_0 + \mathcal{O}( \hbar^2), \\
VS_0 \tilde{U}_0 + \hbar VR_1 \tilde{U}_0 & = S_0 V \tilde{U}_0 + \hbar S_1 V \tilde{U}_0 + \mathcal{O}(\hbar^2), \\
VS_0  - S_0 V  & = \hbar S_1 V - \hbar VR_1 + \mathcal{O}(\hbar^2), \\
[V, S_0] & = \hbar S_1 V - \hbar V R_1 + \mathcal{O}(\hbar^2).
\end{align*}
In terms of principal symbols, this means:
$$ \dfrac{1}{i} \lbrace v ( \tilde{\theta}, \tilde{I}), s_0 ( \tilde{\theta}, \tilde{I}) \rbrace = v( \tilde{\theta}, \tilde{I}) (s_1( \tilde{\theta}, \tilde{I})-r_1( \tilde{\theta}, \tilde{I})) \quad \text{for $( \tilde{\theta}, \tilde{I}) \in (T^* \mathbb{S}^1)_{ \mathbb{C}}$} .$$
Let $v = e^{ia}$, then we have:
$$ \dfrac{1}{i} \lbrace v, s_0 \rbrace = \dfrac{1}{i} \left( \dfrac{\partial e^{ia}}{\partial \tilde{\theta}} \dfrac{\partial s_0}{\partial \tilde{I}} - \dfrac{\partial e^{ia}}{\partial \tilde{I}} \dfrac{\partial s_0}{\partial \tilde{\theta}} \right) = \dfrac{1}{i} i e^{ia} \left( \dfrac{\partial a}{\partial \tilde{\theta}} \dfrac{\partial s_0}{\partial \tilde{I}} - \dfrac{\partial a}{\partial \tilde{I}} \dfrac{\partial s_0}{\partial \tilde{\theta}} \right) = e^{ia} \lbrace a, s_0 \rbrace.$$
Therefore $e^{ia} \lbrace a, s_0 \rbrace = e^{ia}(s_1-r_1)$, \textit{i.e.} $ \lbrace a, s_0 \rbrace = s_1 - r_1$. Moreover, we know that $s_0 = g^{ \epsilon}( \tilde{I})$, so:
\begin{equation} \label{eq_derivee_de_a}
s_1-r_1 = \lbrace a, g^{ \epsilon}( \tilde{I}) \rbrace = \dfrac{\partial a}{\partial \tilde{\theta}} \dfrac{dg^{ \epsilon}}{d \tilde{I}} \quad \text{\textit{i.e.}} \quad \dfrac{\partial a}{\partial \tilde{\theta}} = \left( \dfrac{d g^{ \epsilon}}{d \tilde{I}} \right)^{-1} (s_1-r_1) .
\end{equation}
Since $ \lbrace a, s_0 \rbrace = s_1 - r_1$, then:
\begin{equation} \label{eq_s1_integral_r1}
s_1  = \dfrac{1}{2 \pi} \int r_1 d \tilde{\theta}.
\end{equation}
Consequently, we can determine $s_1$ by using Equation \eqref{eq_s1_integral_r1} and $\partial_{\tilde{ \theta}} a$ by using Equation \eqref{eq_derivee_de_a}. Then, since $ \int \partial_{ \tilde{\theta}} a d \tilde{\theta} = 0$, we can well-define $a$. \\
Thus, we obtain:
$$ (V \tilde{U}_0) \tilde{P}^{ \epsilon}_{ \hbar} := \left( g^{ \epsilon} \left( \dfrac{\hbar}{i} \dfrac{\partial}{\partial \tilde{ \theta}} \right) + \hbar g_1^{ \epsilon} \left( \dfrac{\hbar}{i} \dfrac{\partial}{\partial \tilde{\theta}} \right) + \hbar^2 R_2 \right) (V \tilde{U}_0 ).$$
We then reiterate this process with the operator $W = Id + \hbar V$. This iterative procedure yields the result.
\end{proof}

\subsubsection{Spectrum} $ $\\
\noindent \texttt{Notation:} We denote by $S^{ \epsilon}_{ \hbar}$ the operator acting on $L^2_{ \Flo}( \mathbb{S}^1)$ of symbol $g^{ \epsilon}_{ \hbar} ( \tilde{I})$, \textit{i.e.} $\tilde{U} \tilde{P}^{ \epsilon}_{ \hbar} = S^{ \epsilon}_{ \hbar}\tilde{U} + \mathcal{O}(\hbar^{ \infty}) $ according to Proposition \ref{prop_tilde(UPU)=S} (where $ \tilde{P}_{ \hbar}^{ \epsilon} = B^* P_{ \hbar}^{ \epsilon} B$). \\

\noindent First, we have the following results.

\begin{prop} \label{spectre_de_S}
The spectrum of the operator $S^{ \epsilon}_{ \hbar}$ is given by:
$$ \sigma(S^{ \epsilon}_{ \hbar}) = \lbrace g^{ \epsilon}_{ \hbar}(\hbar k-J), k \in \mathbb{Z} \rbrace,$$
where $g_{ \hbar}^{ \epsilon}$ is the function given by Proposition \ref{prop_tilde(UPU)=S}.
\end{prop}

\begin{proof}
The family $(e_l( \tilde{\theta}))_{l \in \mathbb{Z}} = (e^{il \tilde{\theta}} e^{-(i/\hbar)J \tilde{\theta}})_{l \in \mathbb{Z}}$ for $ \tilde{\theta} \in [0, 2 \pi]$ is an orthonormal basis of the space $L^2_{\Flo}( \mathbb{S}^1)$.
\end{proof}

\begin{prop} \label{prop_spectre_P_tildeP}
Let $P^{ \epsilon}_{ \hbar}$ and $ \tilde{P}^{ \epsilon}_{ \hbar}$ be the operators previously defined. Then $  \sigma(P^{ \epsilon}_{ \hbar}) = \sigma( \tilde{P}^{ \epsilon}_{ \hbar})$.
\end{prop}

\begin{proof}
There exists some unitary operator $B^*$ such that $\tilde{P}^{ \epsilon}_{ \hbar} = B^* P^{ \epsilon}_{ \hbar} B$, therefore the spectrum of the operator $\tilde{P}^{ \epsilon}_{ \hbar}$ is equal to the spectrum of the operator $P^{ \epsilon}_{ \hbar}$.
\end{proof}

We want to describe the spectrum of the operator $ \tilde{P}^{ \epsilon}_{ \hbar}$ by using the spectrum of the operator $S^{ \epsilon}_{ \hbar}$ that we know explicitly. To do so, we follow the method used in \cite{MR2036816,MR2003421} except that in our case the operator $S^{ \epsilon}_{ \hbar}$ obtained by conjugacy from $\tilde{P}^{ \epsilon}_{ \hbar}$ is easier to manipulate. \\
More precisely, we want to describe the spectrum of the operator $ \tilde{P}^{ \epsilon}_{ \hbar}$ in a rectangle of the form:
$$ R_{C, \epsilon_0} = \left\lbrace z \in \mathbb{C}; | \Re (z) -E_0| <  \dfrac{1}{C}, | \Im (z) | < \dfrac{\epsilon_0}{C} \right\rbrace ,$$
for $E_0 \in \mathbb{R}$, where $ \epsilon_0 >0$ is a sufficiently small fixed real number and $C>0$ is a constant. Therefore, we will use some microlocal analysis in a neighbourhood of $ \lbrace \tilde{p}^{ \epsilon}_{ \hbar} = E \rbrace$ where $E \in R_{C, \epsilon_0}$. \\

\noindent \texttt{Notation:}
\begin{itemize}
\item $ \tilde{ \Lambda}_E^{ \epsilon} = \lbrace (y, \eta) \in \tub(\Lambda_{ \Phi_2}); \tilde{p}^{ \epsilon}(y, \eta) = p^{ \epsilon} \circ \overline{\kappa}_{ \phi_1}^{-1}(y, \eta) = E \rbrace$ where $ \tub( \Lambda_{ \Phi_2})$ denotes a tubular neighbourhood of $ \Lambda_{ \Phi_2}$ in $ (T^* \mathbb{S}^1)_{ \mathbb{C}}$;
\item let $I_0 \in \mathbb{R}$ be the constant such that $ (\overline{\kappa}_{ \phi_1}^{-1} \circ \tilde{\kappa}) (\tilde{\Lambda}_{E_0}) = \lbrace \tilde{I} = I_0  \rbrace$ where:
$$ \tilde{\Lambda}_{E_0} = \lbrace ( y, \eta) \in \Lambda_{ \Phi_2}; \left. \tilde{p}^{ \epsilon}(y, \eta) \right|_{ \epsilon = 0} =E_0 \rbrace.$$
\end{itemize}
$ $ \\
We consider the set of quasi-eigenvalues for the operator $S^{ \epsilon}_{ \hbar}$, namely:
$$ \Sigma( \epsilon, \hbar) = \lbrace g^{ \epsilon}_{ \hbar}(\hbar k- J), k \in \mathbb{Z} \rbrace \cap R_{C, \epsilon_0} .$$
First, we can estimate the distance between two elements of the set $ \Sigma( \epsilon, \hbar)$; indeed, let $z=g^{ \epsilon}_{ \hbar}(\hbar k- J)$ and $ \tilde{z}=g^{ \epsilon}_{ \hbar}(\hbar l- J)$ with $k, l \in \mathbb{Z}$ and $k \neq l$. We assume that $z, \tilde{z} \in \Sigma( \epsilon, \hbar )$. \\
Then:
$$ |z- \tilde{z}| \geq \dfrac{\hbar |k-l|}{\mathcal{O}(1)}.$$
\noindent Let:
$$ \delta_{ \hbar} = \dfrac{1}{4} \inf_{k \neq l} \dist(g^{ \epsilon}_{ \hbar}(\hbar k- J), g^{ \epsilon}_{ \hbar}(\hbar l- J)) > \dfrac{\hbar}{\mathcal{O}(1)} ;$$
and consider a family of open discs of the form:
$$ \Omega_k(\hbar) = \lbrace z \in R_{C, \epsilon_0}; |z-g^{ \epsilon}_{ \hbar}(\hbar k- J)| < \delta_{ \hbar} \rbrace .$$

\begin{remark}
The sets $ \Omega_k(\hbar)$ are disjoints (because the distance between two elements of the set $ \Sigma( \epsilon, \hbar)$ is greater than $ \delta_{ \hbar}$).
\end{remark}

We want to show that the spectrum of the operator $ \tilde{P}^{ \epsilon}_{ \hbar}$ in the rectangle $R_{C, \epsilon_0}$ is contained in  the union of discs $\Omega_k(\hbar)$. Therefore, we consider the following equation for $z \in R_{C, \epsilon_0}$:
\begin{equation} \label{eq_(tilde(P)-z)u=v}
( \tilde{P}^{ \epsilon}_{ \hbar} - z) u = v \quad \text{with $u, v \in H( \mathbb{S}^1 + i \mathbb{R}, \Phi_2)$}.
\end{equation}

\noindent First, outside a small neighbourhood of $ \tilde{\Lambda}_E^{ \epsilon} $ in $ \tub(\Lambda_{ \Phi_2})$, there exists a constant $C>0$ such that:
$$
 |\tilde{p}^{ \epsilon}_{ \hbar}( y, \eta)-E| > \dfrac{1}{C}.$$
Indeed, by definition, we have:
\begin{align*}
\tilde{\Lambda}_E^{ \epsilon} & = \lbrace (y, \eta) \in \tub(\Lambda_{ \Phi_2}); \tilde{p}^{ \epsilon}(y, \eta) =E \rbrace, \\
& = \left\lbrace (y, \eta) \in \tub(\Lambda_{ \Phi_2}); \left| \tilde{p}^{ \epsilon}(y, \eta)-E \right|=0 \right\rbrace;
\end{align*}
So, for $(y, \eta) \notin V$ a small neighbourhood of $ \tilde{\Lambda}_E^{ \epsilon} $ in $ \tub( \Lambda_{ \Phi_2})$, there exists a constant $C_1 > 0$ such that:
$$ | \tilde{p}^{ \epsilon}(y, \eta)-E| > \dfrac{1}{C_1} .$$
Besides, we can deduce from Assumption (C'), that for $(y,  \eta) \in \tub( \Lambda_{\Phi_2})$, we have:
$$ \left| \tilde{p}^{ \epsilon}(y, \eta) \right| \geq \dfrac{1}{C} \tilde{m}( \Re( \eta)), \quad \text{for $ |y, \eta| \geq C$} .$$
Let $K = \lbrace (y, \eta) \in \tub( \Lambda_{ \Phi_2}); |(y, \eta)| \leq  C \rbrace$. We assume that $K$ is such that for $(y, \eta) \notin K$, we have $|E| \leq \dfrac{1}{2C}\tilde{m}( \Re( \eta))$. We distinguish two cases:
\begin{itemize}
\item either $(y, \eta) \notin V$ and $(y, \eta) \in K$, then by continuity:
$$ \dfrac{| \tilde{p}^{ \epsilon}(y, \eta) -E|}{\tilde{m}( \Re( \eta))} \neq 0, $$
so there exist a constant $C>0$ such that:
$$ \dfrac{| \tilde{p}^{ \epsilon}(y, \eta) -E|}{\tilde{m}( \Re( \eta))} \geq \dfrac{1}{C} \quad \textit{i.e.} \quad | \tilde{p}^{ \epsilon}(y, \eta) -E| \geq \dfrac{1}{C} \tilde{m}( \Re( \eta)).$$
\item or $(y, \eta) \notin V$ and $(y, \eta) \notin K$, then by the ellipticity condition, we have:
$$ \left| \tilde{p}^{ \epsilon}(y, \eta) \right| \geq \dfrac{1}{C} \tilde{m}( \Re( \eta)),$$
then:
$$ \left| \tilde{p}^{ \epsilon}(y, \eta) -E \right| \geq \left| \tilde{p}^{ \epsilon}(y, \eta) \right| - |E| \geq \dfrac{1}{C} \tilde{m}( \Re( \eta)) - \dfrac{1}{2C} \tilde{m}( \Re( \eta)) = \dfrac{1}{2C} \tilde{m}( \Re( \eta)).$$
\end{itemize}
Consequently, for $(y, \eta) \notin V$, there exist a constant $C>0$ such that:
$$ \left| \tilde{p}^{ \epsilon}(y, \eta)-E \right| \geq \dfrac{1}{2C} \tilde{m}( \Re( \eta)) .$$
We denote by $ \tilde{r}^{ \epsilon}_{ \hbar}$ the function such that $ \tilde{p}^{ \epsilon}_{ \hbar}(y, \eta) = \tilde{p}^{ \epsilon}(y, \eta) + \hbar \tilde{r}^{ \epsilon}_{ \hbar}(y, \eta)$.
Therefore, for $(y, \eta) \notin V$, we have:
\begin{align*}
| \tilde{p}_{ \hbar}^{ \epsilon}(y, \eta) -E| & =  | \tilde{p}_{ \hbar}^{ \epsilon}(y, \eta) - \tilde{p}^{ \epsilon}(y, \eta) + \tilde{p}^{ \epsilon}(y, \eta)-E |, \\
& \geq |\tilde{p}^{ \epsilon}(y, \eta)-E| - | \tilde{p}_{ \hbar}^{ \epsilon}(y, \eta) - \tilde{p}^{ \epsilon}(y, \eta)|, \\
& \geq  \dfrac{1}{2C} \tilde{m}( \Re( \eta)) - \left| \hbar r^{ \epsilon}_{ \hbar}(y, \eta) \right|, \\
& \geq \dfrac{1}{2C} \tilde{m}( \Re( \eta)) - \hbar C \tilde{m}( \Re( \eta)), \quad \text{according to Assumption (B')} \\
& = \left( \dfrac{1}{2C} - \hbar C \right) \tilde{m}( \Re( \eta)).
\end{align*}
Therefore, for $(y, \eta) \notin V$, there exist a constant $C>0$ such that:
$$ |\tilde{p}_{ \hbar}^{ \epsilon}( y, \eta)-E| > \dfrac{1}{C} .$$
Consequently, for $(y, \eta) \notin V$, there exist a constant $C>0$ such that:
$$ \left| \Re( \tilde{p}_{ \hbar}^{ \epsilon}(y, \eta)-E)\right| > \dfrac{1}{C} \quad \text{or} \quad  \left| \Im( \tilde{p}_{ \hbar}^{ \epsilon}(y, \eta)-E)\right| > \dfrac{1}{C}.$$

\noindent \texttt{Notation:} Let $a \in \mathcal{C}^{ \infty}_b( \tub( \Lambda_{ \Phi_2}))$. We denote by $\Op_{ \Phi_2}(a)$ the quantization of the symbol $a$ defined, for $u \in H( \mathbb{S}^1 + i \mathbb{R}, \Phi_2)$, by:
$$ \Op_{ \Phi_2}(a) u(x) = \dfrac{1}{2 \pi \hbar} \int \! \! \! \int_{ \Gamma(x)} e^{(i/\hbar)(x-y) \eta} a \left( \dfrac{x+y}{2}, \eta \right) u(y) dy d \eta ,$$
where $ \Gamma(x) = \left\lbrace (y, \eta) \in (T^* \mathbb{S}^1)_{ \mathbb{C}}; \eta = \dfrac{2}{i} \dfrac{\partial \Phi_2}{\partial x} \left( \dfrac{x+y}{2} \right) \right\rbrace$ (so $ \left( \dfrac{x+y}{2}, \eta \right) \in \Lambda_{\Phi_2}$). \\
Recall that: $\Op_{ \Phi_2}(a) : H( \mathbb{S}^1 + i \mathbb{R}, \Phi_2) \longrightarrow H( \mathbb{S}^1 + i \mathbb{R}, \Phi_2)$. \\

\noindent Let $X_{ \left. p^{ \epsilon} \right|_{ \epsilon=0}}$ be the flow of the Hamiltonian vector field associated to $ \left. p^{ \epsilon} \right|_{ \epsilon =0 }$. Let $X_{ \left. \tilde{p}^{ \epsilon} \right|_{ \epsilon=0}}$ be the image by the function $ \overline{\kappa}_{ \Phi_1}$ of the real flow of $X_{ \left. p^{ \epsilon} \right|_{ \epsilon=0}}$. We consider a partition of unity on the manifold $ \tub(\Lambda_{ \Phi_2})$:
$$ 1 = \chi + \psi_1^+ + \psi_1^- + \psi_2^+ + \psi_2^-,$$
with:
\begin{enumerate}
\item[1.] $\chi \in \mathcal{C}^{ \infty}_0( \tub(\Lambda_{ \Phi_2}))$ a smooth function such that $ \chi = 1$ in a neighbourhood of $ \tilde{\Lambda}_E^{ \epsilon}$ and such that its support is contained in a small neighbourhood of $ \tilde{\Lambda}_{E}^{ \epsilon} $ where: $\tilde{U} \tilde{P}^{ \epsilon}_{ \hbar} \Op_{ \Phi_2}( \chi) =S^{ \epsilon}_{ \hbar} \tilde{U} \Op_{ \Phi_2}( \chi) + \mathcal{O}(\hbar^{ \infty})$;
\item[2.] $ \psi_1^{\pm} \in \mathcal{C}_0^{ \infty}((T^* \mathbb{S}^1)_{ \mathbb{C}})$ a smooth function supported in a region invariant under the flow of $X_{ \left. \tilde{p}^{ \epsilon} \right|_{ \epsilon=0}}$ and where:
$$ \Im( \tilde{p}_{ \hbar}^{ \epsilon}-E) > \pm \dfrac{1}{C};$$
\item[3.] $ \psi_2^{\pm} \in \mathcal{C}^{ \infty}_b( \tub( \Lambda_{ \Phi_2}))$ a smooth function supported in a region where:
$$ \Re( \tilde{p}_{ \hbar}^{ \epsilon} -E) > \pm \dfrac{1}{C} .$$
\end{enumerate}
Moreover, we can choose the functions $ \psi_1^{ \pm}$ such that their Poisson brackets commute with $ \left. \tilde{p}^{ \epsilon} \right|_{ \epsilon =0 }$. \\

\noindent To show the pertinence of this partition of unity, we are going to look at some properties where it intervenes. The proofs of these propositions are similar to what is done in \cite{MR2036816}, thus we do not recall them here.

\begin{prop} \label{prop_norme_(1-chi)u}
Let $u, v \in H( \mathbb{S}^1 + i \mathbb{R}, \Phi_2)$ satisfying Equation \eqref{eq_(tilde(P)-z)u=v}. Then, we have:
$$ \| \Op_{ \Phi_2}(1- \chi)u \| \leq \mathcal{O} \left( 1 \right) \| v \| + \mathcal{O}(\hbar^{ \infty}) \| u \| .$$
\end{prop}

\noindent Then, from Equation \eqref{eq_(tilde(P)-z)u=v}, we have:
\begin{equation} \label{eq(tilpe(P)-z)chi u}
( \tilde{P}^{ \epsilon}_{ \hbar}-z) \Op_{ \Phi_2}(\chi) u  = \Op_{ \Phi_2}(\chi) v + w \quad \text{with $w = [ \tilde{P}^{ \epsilon}_{ \hbar}, \Op_{ \Phi_2}( \chi)]u$}.
\end{equation}
Since $w$ is microlocalized in the support of $[ \tilde{P}_{ \hbar}^{ \epsilon}, \Op_{ \Phi_2}( \chi)]$, which is contained outside a small neighbourhood of $ \tilde{\Lambda}_E^{ \epsilon}$, we can show using Proposition \ref{prop_norme_(1-chi)u} that:
\begin{equation} \label{eq_w}
\| w \| \leq \mathcal{O} \left( \hbar \right) \| v \| + \mathcal{O}(\hbar^{ \infty}) \| u \|.
\end{equation}

\noindent By applying the operator $ \tilde{U}$ on Equation \eqref{eq(tilpe(P)-z)chi u}, we obtain:
\begin{align*}
& \tilde{U} (( \tilde{P}^{ \epsilon}_{ \hbar}-z) \Op_{ \Phi_2}(\chi) u ) = \tilde{U}( \Op_{ \Phi_2}(\chi) v +w), \\
& \tilde{U} \tilde{P}^{ \epsilon}_{ \hbar} \Op_{ \Phi_2}(\chi) u - z \tilde{U} \Op_{ \Phi_2}(\chi) u = \tilde{U} \Op_{ \Phi_2}(\chi) v + \tilde{U} w, \\
& (S^{ \epsilon}_{ \hbar} \tilde{U} + \mathcal{O}(\hbar^{ \infty})) \Op_{ \Phi_2}(\chi) u - z \tilde{U} \Op_{ \Phi_2}(\chi) u =  \tilde{U} \Op_{ \Phi_2}(\chi) v + \tilde{U} w,
\end{align*}
because $ \tilde{U} \tilde{P}^{ \epsilon}_{ \hbar} \Op_{ \Phi_2} ( \chi) = S^{ \epsilon}_{ \hbar} \tilde{U} \Op_{ \Phi_2}( \chi) + \mathcal{O}(\hbar^{ \infty})$ by definition of the partition of unity. Therefore, we have:
\begin{equation} \label{eq3}
(S^{ \epsilon}_{ \hbar}-z) \tilde{U} \Op_{ \Phi_2}(\chi) u = \tilde{U} \Op_{ \Phi_2}(\chi) v + \tilde{U} w + T_{ \infty} u,
\end{equation} 
where $T_{ \infty} = \mathcal{O}(\hbar^{ \infty}):H( \mathbb{S}^1 + i \mathbb{R}, \Phi_2) \longrightarrow L^2_{\Flo}( \mathbb{S}^1)$. \\

\noindent From the explicit definition of the operator $S^{ \epsilon}_{ \hbar}$ we see that, if $z \in R_{C, \epsilon_0} \setminus \bigcup \Omega_k(\hbar)$, the operator $S^{ \epsilon}_{ \hbar}-z: L^2_{\Flo}( \mathbb{S}^1) \longrightarrow L^2_{\Flo}( \mathbb{S}^1)$ is microlocally invertible in the region where $ | \tilde{I} - I_0| \leq \dfrac{1}{\mathcal{O}(1)}$ (which corresponds to the domain where the operator $S_{ \hbar}^{ \epsilon}$ is well-defined) and its microlocal inverse is of the norm $ \mathcal{O} \left( \dfrac{1}{ \hbar} \right)$. Moreover, we also have the following proposition.

\begin{prop} \label{prop_norme_chiu}
Let $z \in R_{C, \epsilon_0} \setminus \bigcup \Omega_k(\hbar)$. Let $u, v \in H( \mathbb{S}^1 + i \mathbb{R}, \Phi_2)$ satisfying Equation \eqref{eq3}.
Then, we have the following estimate:
$$ \| \Op_{ \Phi_2}(\chi) u \| \leq \dfrac{\mathcal{O}(1)}{\hbar} \| v \| + \mathcal{O}(\hbar^{ \infty}) \| u \|.$$
\end{prop}

\begin{proof}
We multiply Equation \eqref{eq3} by $ \tilde{U}^{-1}(S^{ \epsilon}_{ \hbar}-z)^{-1}$ (where $ \tilde{U}^{-1}$ is the microlocal inverse of $ \tilde{U}$ which exists in the domain of the function $ \chi$) and use the estimate on the norm of the operator $S^{ \epsilon}_{ \hbar}-z$, the estimate on $w$ and the definition of $T_{ \infty}$.
\end{proof}

\noindent We deduce from Propositions \ref{prop_norme_(1-chi)u} and \ref{prop_norme_chiu}, that if $z \in R_{C, \epsilon_0} \setminus \bigcup \Omega_k(\hbar)$, then the operator $ \tilde{P}^{ \epsilon}_{ \hbar}-z: H( \mathbb{S}^1 + i \mathbb{R}, \Phi_2) \longrightarrow H( \mathbb{S}^1 + i \mathbb{R}, \Phi_2)$ is injective. \\
Besides the operator $ \tilde{P}^{ \epsilon}_{ \hbar}-z : H( \mathbb{S}^1 + i \mathbb{R}, \Phi_2, \tilde{m}) \longrightarrow H( \mathbb{S}^1 + i \mathbb{R}, \Phi_2)$ is also Fredholm of index $0$ (\textit{i.e.} it is an operator with finite-dimensional kernel and cokernel whose dimensions are the same). Namely, by the ellipticity of the principal symbol $\tilde{p}^{ \epsilon}$, we can construct an inverse for $\tilde{P}^{ \epsilon}_{ \hbar}-z+K$ where $K$ is a compact operator. Therefore, we obtain that $ \tilde{P}^{ \epsilon}_{ \hbar}-z+K$ is Fredholm of index $0$ and that proves the fact that the operator $ \tilde{P}_{ \hbar}^{ \epsilon}-z$ is also Fredholm of index $0$. \\
Therefore, if $z \in R_{C, \epsilon_0} \setminus \bigcup \Omega_k(\hbar)$ we obtain that:
$$ \tilde{P}^{ \epsilon}_{ \hbar}-z: H(\mathbb{S}^1 + i \mathbb{R}, \Phi_2, \tilde{m}) \longrightarrow H( \mathbb{S}^1 + i \mathbb{R}, \Phi_2),$$
is bijective. \\

We can sum up what we have done so far by saying that  the eigenvalues of the operator $ \tilde{P}^{ \epsilon}_{ \hbar}$ in $R_{C, \epsilon_0}$ are localized in the open discs $ \Omega_k(\hbar)$. We are now focusing on one of these discs.\\
Since the eigenfunctions are microlocalized in a neighbourhood of $ \tilde{I}= I_0$, then we consider the couples $(\hbar,k)$ such that $z \in \Omega_k( \hbar)$, \textit{i.e.} $|\hbar k-J - I_0| < \dfrac{1}{C}$. \\
We want to prove that $z \in \Omega_k(\hbar)$ is a $ \mathcal{O}( \hbar^{ \infty})$-close to an eigenvalue of the operator $\tilde{P}^{ \epsilon}_{ \hbar}$ if and only if:
$$ z = g^{ \epsilon}_{ \hbar}(\hbar k- J) + \mathcal{O}(\hbar^{ \infty}).$$
To do so, we are going to study two Grushin problems concerning $S^{ \epsilon}_{ \hbar}-z$ and $ \tilde{P}^{ \epsilon}_{ \hbar}-z$ respectively; we recall the definition of this problem (for more details on this linear algebraic tool see \cite{MR2394537}).

\begin{defi}[Grushin problem]
A Grushin problem for an operator $P:H_1 \longrightarrow H_2$ between two Hilbert spaces is a system:
$$
\left\lbrace
\begin{split}
& Pu+R_- u_- = v, \\
& R_+ u = v_+;
\end{split}
\right.$$
where $R_-:H_- \longrightarrow H_2$, $R_+: H_1 \longrightarrow H_+$, with $H_-, H_+$ two Hilbert spaces and where $(u, u_-) \in H_1 \times H_-$, $(v, v_+) \in H_2 \times H_+$. The matrix associated with the Grushin problem is defined by:
$$
\mathcal{P} : = \begin{pmatrix}
P & R_- \\ 
R_+ & 0
\end{pmatrix}  : H_1 \times H_- \longrightarrow H_2 \times H_+.$$
\end{defi}

\noindent First, we consider a Grushin problem for the operator $S^{ \epsilon}_{ \hbar}-z$. This problem is globally defined if we consider the function $g_{ \hbar}^{ \epsilon}$ (defining the operator $S_{ \hbar}^{ \epsilon}$) as a compactly supported one. \\
Let $(e_l)_{l \in \mathbb{Z}}$ be the functions defined for $l \in \mathbb{Z}$ and $ \tilde{\theta} \in [0, 2 \pi]$ by:
$$ e_l( \tilde{\theta}) = e^{(i/\hbar)( \hbar l- J)\tilde{\theta}} = e^{il \tilde{\theta}} e^{-(i/ \hbar) J \tilde{\theta}} .$$
The family of functions $(e_l)_{l \in \mathbb{Z}}$ forms an orthonormal basis of the space $L^2_{\Flo}( \mathbb{S}^1)$. \\
Let $\hat{R}_+$ and $ \hat{R}_-$ be the following operators:
\begin{align*}
\hat{R}_+: L^2_{\Flo}( \mathbb{S}^1) & \longrightarrow \mathbb{C}&  \quad \quad \hat{R}_-: \mathbb{C} & \longrightarrow L^2_{\Flo}( \mathbb{S}^1),\\
u & \longmapsto \langle u | e_k \rangle &  \quad \quad \tilde{u} & \longmapsto \tilde{u} e_k .
\end{align*}
We look at the following Grushin problem, for $(u, \tilde{u})$ and $(v, \tilde{v}) \in L^2_J( \mathbb{S}^1) \times \mathbb{C}$:
$$
\left\lbrace
\begin{split}
& ( S^{ \epsilon}_{ \hbar}-z) u + \hat{R}_- \tilde{u} = v, \\
& \hat{R}_+ u = \tilde{v}.
\end{split}
\right.$$

\begin{prop} \label{prop_pb_grushin_S}
Let:
$$
\mathcal{S} = \begin{pmatrix}
S^{ \epsilon}_{ \hbar}-z & \hat{R}_- \\ 
\hat{R}_+ & 0
\end{pmatrix}: L^2_{ \Flo}( \mathbb{S}^1) \times \mathbb{C} \longrightarrow L^2_{ \Flo}( \mathbb{S}^1) \times \mathbb{C}. $$
Then, the operator $ \mathcal{S}$ admits an inverse defined by:
$$ \hat{\mathcal{E}} = \begin{pmatrix}
\hat{E}(z) & \hat{E}_+ \\ 
\hat{E}_- & \hat{E}_{-,+}(z)
\end{pmatrix}, $$
with:
\begin{enumerate}
\item[1.] $ \hat{E}_+ = \hat{R}_-$;
\item[2.] $ \hat{E}_- = \hat{R}_+$;
\item[3.] $ \hat{E}_{-,+}(z) = z-g^{ \epsilon}_{ \hbar}(\hbar k-J)$;
\end{enumerate}
Furthermore, the components of the operator $ \hat{\mathcal{E}}$ satisfy the following estimates:
\begin{enumerate}
\item[(i)] $ \hat{E} = \dfrac{\mathcal{O}(1)}{ \hbar}: L^2_{\Flo}( \mathbb{S}^1) \longrightarrow L^2_{\Flo}( \mathbb{S}^1)$;
\item[(ii)] $ \hat{E}_+ = \mathcal{O}(1): \mathbb{C} \longrightarrow L^2_{\Flo}( \mathbb{S}^1)$;
\item[(iii)] $ \hat{E}_- = \mathcal{O}(1): L^2_{\Flo}( \mathbb{S}^1) \longrightarrow \mathbb{C}$;
\item[(iv)] $ \hat{E}_{-,+} = \mathcal{O}( \hbar): \mathbb{C} \longrightarrow \mathbb{C}$.
\end{enumerate}
Moreover, for all $( u, \tilde{u}), (v, \tilde{v}) \in L^2_{\Flo}( \mathbb{S}^1) \times \mathbb{C}$ satisfying $ \mathcal{S} (u, \tilde{u}) = (v, \tilde{v})$, we have the following estimate:
\begin{equation} \label{esti_grushin_S_h}
 \hbar \| u \|_{L^2_{\Flo}} + | \tilde{u} | \leq \mathcal{O}(1) ( \| v \|_{L^2_{\Flo}} +  \hbar | \tilde{v} | ).
\end{equation}
\end{prop}

\begin{proof}
We invert the system $ \mathcal{S}(u, \tilde{u}) = (v, \tilde{v})$ with $(u, \tilde{u}), (v, \tilde{v}) \in L^2_{\Flo}( \mathbb{S}^1) \times \mathbb{C}$ by using the orthonormal basis $(e_l)_{l \in \mathbb{Z}}$ and the explicit expression $S^{ \epsilon}_{ \hbar}=g^{ \epsilon}_{ \hbar} \left( \dfrac{\hbar}{i} \dfrac{\partial}{\partial \tilde{\theta}} \right)$ to obtain the expression of $ \hat{\mathcal{E}}$. Then, the estimates \textit{(i)}, \textit{(ii)}, \textit{(iii)} and \textit{(iv)} can be deduced from \textit{1.}, \textit{2.} and \textit{3.} always by using the properties of the basis $(e_l)_{l \in \mathbb{Z}}$. Lastly, the estimate \eqref{esti_grushin_S_h} can be deduced from \textit{(i)}, \textit{(ii)}, \textit{(iii)} and \textit{(iv)}.
\end{proof}

\noindent We now deal with a global Grushin problem for the operator $ \tilde{P}^{ \epsilon}_{ \hbar}-z$.\\
We consider the following operators, for all $(\hbar,k)$ such that $z \in \Omega_k(\hbar)$:
\begin{align*}
R_+: H( \mathbb{S}^1 + i \mathbb{R}, \Phi_2) & \longrightarrow \mathbb{C} ,\\
u & \longmapsto \hat{R}_+ \tilde{U} \Op_{ \Phi_2}(\chi) u := \langle \tilde{U} \Op_{ \Phi_2}(\chi) u | e_k \rangle  
\end{align*}
and:
\begin{align*}
R_-: \mathbb{C} & \longrightarrow H( \mathbb{S}^1 + i \mathbb{R}, \Phi_2), \\
\tilde{u} & \longmapsto \tilde{U}^{-1} \hat{R}_- \tilde{u} := \tilde{u} \tilde{U}^{-1} e_k .
\end{align*}
where $\tilde{U}$ is the operator defined in Proposition \ref{prop_tilde(UPU)=S} such that microlocally :
$$ \tilde{U} \tilde{P}^{ \epsilon}_{ \hbar} \Op_{ \Phi_2}( \chi) = S^{ \epsilon}_{ \hbar} \tilde{U} \Op_{ \Phi_2}( \chi) + \mathcal{O}(\hbar^{ \infty}) ,$$
and where $ \tilde{U}^{-1}$ denote the microlocal inverse of $ \tilde{U}$. \\
First, according to \cite{MR2036816}, notice that we have the following property:
$$ \Op_{ \Phi_2}(\chi) R_- = R_- + \mathcal{O}(\hbar^{ \infty}): \mathbb{C} \longrightarrow H( \mathbb{S}^1 + i \mathbb{R}, \Phi_2) ,$$
up to decreasing the support of the function $ \chi$ if necessary (because the functions $e_k$ and $ \chi$ are localized in the same neighbourhood). \\
We consider the following Grushin problem, for $(u, \tilde{u})$ and $(v, \tilde{v}) \in H( \mathbb{S}^1 + i \mathbb{R}, \Phi_2) \times \mathbb{C}$:
$$
\left\lbrace
\begin{split}
& ( \tilde{P}^{ \epsilon}_{ \hbar}-z) u + R_- \tilde{u} = v, \\
& R_+ u = \tilde{v}.
\end{split}
\right.$$

\begin{prop} \label{prop_pb_grushin_P}
For all $(v, \tilde{v}) \in H( \mathbb{S}^1 + i \mathbb{R}, \Phi_2) \times \mathbb{C}$, this Grushin problem admits a unique solution $(u, \tilde{u}) \in H( \mathbb{S}^1 + i \mathbb{R}, \Phi_2, \tilde{m}) \times \mathbb{C}$ with the following estimate:
\begin{equation} \label{eq_Grushin_P}
 \hbar \| u \| + | \tilde{u} | \leq \mathcal{O}(1) \left( \| v \| + \hbar | \tilde{v} | \right).
\end{equation}
\end{prop}

\begin{proof}
To prove this result, we are going to modify the Grushin problem for the operator $ \tilde{P}^{ \epsilon}_{ \hbar}-z$ and reduce ourselves to that of the operator $S^{ \epsilon}_{ \hbar}-z$, we will then be able to use Proposition \ref{prop_pb_grushin_S}. \\
Indeed, we start by applying the operator $ \Op_{ \Phi_2}( \chi)$ to the first equation of the Grushin problem for $ \tilde{P}_{ \hbar}^{ \epsilon}-z$:
$$
\left\lbrace
\begin{split}
& \Op_{ \Phi_2}( \chi) ( \tilde{P}^{ \epsilon}_{ \hbar}-z) u + \Op_{ \Phi_2}( \chi) R_- \tilde{u} = \Op_{ \Phi_2}( \chi) v, \\
& R_+ u = \tilde{v}.
\end{split}
\right.$$
Since $ \Op_{ \Phi_2}( \chi) R_- = R_- + \mathcal{O}(h^{ \infty}) := R_- - R_-^{ \infty}$, we have:
$$
\left\lbrace
\begin{split}
& \Op_{ \Phi_2}( \chi) ( \tilde{P}^{ \epsilon}_{ \hbar}-z) u + R_- \tilde{u} = \Op_{ \Phi_2}( \chi) v + R_-^{ \infty} \tilde{u}, \\
& R_+ u = \tilde{v}.
\end{split}
\right.$$
Since:
\begin{align*}
\Op_{ \Phi_2}( \chi) ( \tilde{P}^{ \epsilon}_{ \hbar} -z) u & = \Op_{ \Phi_2}( \chi) \tilde{P}^{ \epsilon}_{ \hbar} u - z \Op_{ \Phi_2}( \chi) u, \\
& = \tilde{P}^{ \epsilon}_{ \hbar} \Op_{ \Phi_2}( \chi) u - [ \tilde{P}^{ \epsilon}_{ \hbar}, \Op_{ \Phi_2}( \chi)] u - z \Op_{ \Phi_2}( \chi) u, \\
& = ( \tilde{P}^{ \epsilon}_{ \hbar} - z) \Op_{ \Phi_2}( \chi) u - [ \tilde{P}^{ \epsilon}_{ \hbar}, \Op_{ \Phi_2} ( \chi)] u,
\end{align*}
then, if $w := [ \tilde{P}^{ \epsilon}_{ \hbar}, \Op_{ \Phi_2} ( \chi)] u$, we have:
$$
\left\lbrace
\begin{split}
&  ( \tilde{P}^{ \epsilon}_{ \hbar}-z) \Op_{ \Phi_2}( \chi) u + R_- \tilde{u} = \Op_{ \Phi_2}( \chi) v + w + R_-^{ \infty} \tilde{u}, \\
& R_+ u = \tilde{v}.
\end{split}
\right.$$
where $w$ satisfies the following estimate:
$$ \| w \| \leq \mathcal{O} \left(\hbar \right) \| v \| + \mathcal{O}(h^{ \infty}) ( \| u \| + | \tilde{u} |) .$$
We apply the operator $ \tilde{U}$ to the first equation:
$$
\left\lbrace
\begin{split}
& \tilde{U} ( \tilde{P}^{ \epsilon}_{ \hbar}-z) \Op_{ \Phi_2}( \chi) u + \tilde{U} R_- \tilde{u} = \tilde{U} \Op_{ \Phi_2}( \chi) v + \tilde{U} w + \tilde{U} R_-^{ \infty} \tilde{u}, \\
& R_+ u = \tilde{v}.
\end{split}
\right.$$
Besides, since $ R_- = \tilde{U}^{-1} \hat{R}_-$ and $ \tilde{U} ( \tilde{P}^{ \epsilon}_{ \hbar} -z) \Op_{ \Phi_2}( \chi) = (S^{ \epsilon}_{ \hbar} -z) \tilde{U} \Op_{ \Phi_2}( \chi) + \mathcal{O}(h^{ \infty})$, the system becomes:
$$
\left\lbrace
\begin{split}
& ( S^{ \epsilon}_{ \hbar}-z) \tilde{U} \Op_{ \Phi_2}( \chi) u + \hat{R}_- \tilde{u} = \tilde{U} \Op_{ \Phi_2}( \chi) v + \tilde{U} w + \tilde{w}, \\
& R_+ u = \tilde{v}.
\end{split}
\right.$$
where $ \tilde{w}$ satisfies $ | \tilde{w} | \leq \mathcal{O}(h^{ \infty})( \| u \| + | \tilde{u} |)$. \\
Moreover by definition $ R_+ = \hat{R}_+ \tilde{U} \Op_{ \Phi_2}( \chi)$, then the system can be written as:
$$
\left\lbrace
\begin{split}
& ( S^{ \epsilon}_{ \hbar}-z) \tilde{U} \Op_{ \Phi_2}( \chi) u + \hat{R}_- \tilde{u} = \tilde{U} \Op_{ \Phi_2}( \chi) v + \tilde{U} w + \tilde{w}, \\
& \hat{R}_+ \tilde{U} \Op_{ \Phi_2}( \chi) u = \tilde{v}.
\end{split}
\right.$$
We recognize the Grushin problem for $S^{ \epsilon}_{ \hbar}-z$ and therefore we deduce our result.\\
The proof of the estimate \eqref{eq_Grushin_P} uses the estimate \eqref{esti_grushin_S_h} and the estimations on the norm of $w$ and $ \tilde{w}$.
\end{proof}

\noindent \noindent Let:
$$
\mathcal{P} = \begin{pmatrix}
\tilde{P}^{ \epsilon}_{ \hbar}-z & R_- \\ 
R_+ & 0
\end{pmatrix} :H(\mathbb{S}^1 + i \mathbb{R}, \Phi_2) \times \mathbb{C} \longrightarrow H( \mathbb{S}^1 + i \mathbb{R}, \Phi_2) \times \mathbb{C} . $$
Then, according to Proposition \ref{prop_pb_grushin_P}, the operator $ \mathcal{P}$ is injective for $z \in \Omega_k( \hbar)$ and because it is a rank-one perturbation of a Fredholm operator of index $0$, we know that the operator $ \mathcal{P}$ is bijective for $z \in \Omega_k( \hbar)$. \\
We denote the inverse of $ \mathcal{P}$ by:
$$
\mathcal{E} = \begin{pmatrix}
E(z) & E_+ \\ 
E_- & E_{-,+}(z)
\end{pmatrix}, $$
and recall that the spectrum of $ \tilde{P}^{ \epsilon}_{ \hbar}$ in $ \Omega_k(\hbar)$ is equal to the set of $z \in \mathbb{C}$ such that $E_{-,+}(z) = 0$. Therefore, we want to determine the component $E_{-,+}(z)$.

\begin{cor} \label{prop_E+_E-+}
The components of the operator $ \mathcal{E}$ are given by:
\begin{enumerate}
\item[1.] $ E_+ = \tilde{U}^{-1} \hat{E}_+ + \mathcal{O}(\hbar^{ \infty})$;
\item[2.] $ E_{-,+}(z) = \hat{E}_{-,+}(z) + \mathcal{O}(\hbar^{ \infty}) = z-g^{ \epsilon}_{ \hbar}(\hbar k-J) + \mathcal{O}(\hbar^{ \infty})$.
\end{enumerate}
\end{cor}

\begin{proof}
Since $ \mathcal{P} \mathcal{E} = Id$, we have:
$$
\left\lbrace
\begin{split}
& ( \tilde{P}^{ \epsilon}_{ \hbar} -z) E_+ + R_- E_{-,+}(z) =0, \\
& R_+ E_+ = 1.
\end{split}
\right.$$
Therefore, we need to show, that up to $ \mathcal{O}(h^{ \infty})$, we have:
$$
\left\lbrace
\begin{split}
& ( \tilde{P}^{ \epsilon}_{ \hbar} -z) \tilde{U}^{-1} \hat{E}_+ + R_- \hat{E}_{-,+}(z) \equiv 0, \\
& R_+ \tilde{U}^{-1} \hat{E}_+ \equiv 1.
\end{split}
\right.$$
We have:
\begin{align*}
& R_+ \tilde{U}^{-1} \hat{E}_+ \\
&  = \hat{R}_+ \tilde{U} \Op_{ \Phi_2}( \chi) \tilde{U}^{-1} \hat{E}_+ \quad \text{by definition of $ R_+$}, \\
& \equiv 1 \quad \text{by definition of $ \chi$ and because $ \mathcal{S} \hat{\mathcal{E}}=Id$ according to Proposition \ref{prop_pb_grushin_S}}.
\end{align*}
Then, we have:
\begin{align*}
& ( \tilde{P}^{ \epsilon}_{ \hbar} -z) \tilde{U}^{-1} \hat{E}_+ + R_- \hat{E}_{-,+}(z) \\
&  = ( \tilde{P}^{ \epsilon}_{ \hbar} -z) \tilde{U}^{-1} \hat{E}_+ + \tilde{U}^{-1} \hat{R}_- \hat{E}_{-,+}(z) \quad \text{by definition of $R_-$}, \\
& \equiv \tilde{U}^{-1} (S^{ \epsilon}_{ \hbar} -z) \hat{E}_+ + \tilde{U}^{-1} \hat{R}_- \hat{E}_{-,+}(z) \quad \text{because $ \tilde{U} \tilde{P}^{ \epsilon}_{ \hbar} = S^{ \epsilon}_{ \hbar} \tilde{U} + \mathcal{O}(h^{ \infty})$}, \\
& \equiv 0 \quad \text{because $\mathcal{S} \hat{\mathcal{E}}=Id$}.
\end{align*}
\end{proof}

\noindent We can sum up what we have done so far by the following proposition.

\begin{prop} \label{prop_spectre_tildeP}
Let $ \tilde{P}^{ \epsilon}_{ \hbar}:H( \mathbb{S}^1 + i \mathbb{R}, \Phi_2) \longrightarrow H(\mathbb{S}^1 + i \mathbb{R}, \Phi_2)$ be the operator previously defined. 
Then, with the notations of Proposition \ref{prop_tilde(UPU)=S}, we have:
$$ \sigma( \tilde{P}^{ \epsilon}_{ \hbar}) \cap R_{C, \epsilon_0} = \sigma(S^{ \epsilon}_{ \hbar}) \cap R_{C, \epsilon_0} + \mathcal{O}(\hbar^{ \infty}) = \lbrace g^{ \epsilon}_{ \hbar}(\hbar k-J) , k \in \mathbb{Z} \rbrace \cap R_{C, \epsilon_0} + \mathcal{O}( \hbar^{ \infty}) .$$
\end{prop}

\begin{proof}
By definition of the spectrum, we know that $z \in \sigma( \tilde{P}^{ \epsilon}_{ \hbar})$ if and only if $ \tilde{P}^{ \epsilon}_{ \hbar}-z$ is non-invertible, \textit{i.e.} $ E_{-,+}(z) = 0$, \textit{i.e.} $z = g^{ \epsilon}_{ \hbar}(\hbar k-J) + \mathcal{O}(\hbar^{ \infty})$ (because $ \tilde{P}^{ \epsilon}_{ \hbar}-z$ is invertible if and only if $E_{-,+}$ is invertible too, \textit{i.e.} if and only if $E_{-,+}(z) \neq 0$).
\end{proof}

\noindent Now, we can conclude and determine the spectrum of the operator $P^{ \epsilon}_{ \hbar}$ by using Propositions \ref{prop_spectre_tildeP} and \ref{prop_spectre_P_tildeP}. This ends the proof of  Theorem \ref{theo_cas_simple}.

\subsection{Proof of Theorem \ref{theo_result}}
To prove Theorem \ref{theo_result}, we are going to make a link with the $L^2( \mathbb{S}^1)$-case, then we will apply techniques developed in the proof of Theorem \ref{theo_cas_simple}. \\
We consider the pseudo-differential operator $P^{ \epsilon}_{ \hbar}$ acting on $L^2( \mathbb{R})$ and depending smoothly on $ \epsilon$ of the form:
$$ P^{ \epsilon}_{ \hbar}(x, \hbar D_x) = F^{ \epsilon}_{ \hbar}(x, \hbar D_x) + i \epsilon Q^{ \epsilon}_{ \hbar}(x, \hbar D_x) ,$$
satisfying the hypotheses (A) to (E) (which were defined in the introduction). \\
We want to obtain Bohr-Sommerfeld quantization conditions for this operator by using Theorem \ref{theo_cas_simple}. Therefore, we are looking for a real canonical transformation of the form:
\begin{align*}
\hat{\kappa}: \Vois(f^{ \epsilon}=cst, \mathbb{R}^2) & \longrightarrow \Vois(I=cst, \mathbb{S}^1 \times \mathbb{R}), \\
(x,\xi) & \longmapsto ( \theta, I),
\end{align*}
and such that:
$$ f^{ \epsilon} \circ \hat{\kappa}^{-1}( \theta,I) = \hat{f}^{ \epsilon} (I),$$
where $ \hat{f}^{ \epsilon}$ is an analytic function depending smoothly on $ \epsilon$. \\
To construct such a canonical transformation $ \hat{\kappa}$, we are going to use the action-angle coordinates theorem. \\
\noindent Let $E_0 \in \mathbb{R}$, we consider:
$$ \Lambda_{E_0}^{ \epsilon} = \lbrace (x, \xi) \in \mathbb{R}^2, f^{ \epsilon}(x, \xi) = E_0 \rbrace ;$$
recall that $ \Lambda_{E_0}$ is compact, connected and regular. \\
Let $ \hat{\gamma}_{E_0}$ be a loop generating $ \pi_1( \Lambda_{E_0}^{ \epsilon})$ and let:
$$ I(E_0) = \dfrac{1}{2\pi} \int_{ \hat{\gamma}_{E_0}} \xi dx. $$
Then, by applying the action-angle coordinates theorem with the parameter $ \epsilon$, we know that there exists a symplectomorphism:
\begin{align*}
\hat{\kappa}: \Vois(f^{ \epsilon}=E_0, \mathbb{R}^2) & \longrightarrow \Vois( I=cst, \mathbb{S}^1 \times \mathbb{R}), \\
(x,  \xi) & \longmapsto ( \theta, I);
\end{align*}
such that:
$$ f^{ \epsilon} \circ \hat{\kappa}^{-1}(\theta, I) = \hat{f}^{ \epsilon}(I) .$$
The canonical transformation $ \hat{\kappa}$ transforms the principal symbol $p^{ \epsilon}(x, \xi)$ to a principal symbol of the form:
$$ p^{ \epsilon} \circ \hat{\kappa}^{-1}( \theta, I) = \hat{f}^{ \epsilon}(I) + i \epsilon \hat{q}^{ \epsilon}( \theta, I) ;$$
where $ \hat{q}^{ \epsilon}( \theta, I) = q^{ \epsilon} \circ \hat{\kappa}^{-1}( \theta, I)$. Therefore, we reduce our problem to the study of a principal symbol on $ \mathbb{S}^1 \times \mathbb{R}$ of the form used in Theorem \ref{theo_cas_simple}.\\

\noindent Moreover, we can choose the transformation $ \hat{\kappa}$ such that for any loop $ \hat{\gamma}$, we have:
$$ \int_{ \hat{\gamma}} \hat{\kappa}^* I d \theta - \xi dx = 0.$$
Indeed, since $ \hat{\kappa}$ is a canonical transformation then:
$$ \hat{\kappa}^* (dI \wedge d \theta) = d \xi \wedge dx .$$
Consequently, the $1$-form $ \hat{\kappa}^* (I d \theta) - \xi dx $ is closed and by Stokes theorem, we obtain that the following integral over a loop $ \hat{\gamma}$:
$$ \int_{ \hat{\gamma}} \hat{\kappa}^* (I d \theta) - \xi dx ,$$
depends only on the homotopy class of $ \hat{\gamma}$, then there exists a real constant $c_{ \hat{\gamma}}( \hat{\kappa})$ such that:
$$ \int_{ \hat{\gamma}} \hat{\kappa}^* (I d \theta) - \xi dx = c_{ \hat{\gamma}}( \hat{\kappa}) ,$$
and we can choose this constant equals to zero (up to change the transformation $ \hat{\kappa}$ if necessary). \\
Besides, we can extend the real canonical transformation $\hat{\kappa}$ into a complex canonical transformation such that:
$$ p^{ \epsilon} \circ \hat{\kappa}^{-1} ( \theta, I) = \hat{f}^{ \epsilon}(I) + i \epsilon \hat{q}^{ \epsilon}( \theta, I), \quad \text{for $( \theta, I)$ complex coordinates}. $$
Consequently, for $ \hat{\gamma}$ a complex loop, the following relation is always true:
$$ \int_{ \hat{\gamma}} \hat{\kappa}^* (I d \theta) - \xi dx = 0 .$$ 
Let, for $C>0$ a constant and for $ \epsilon_0$ a sufficiently small fixed real number:
$$ E \in \left\lbrace z \in \mathbb{C}, | \Re(z) - E_0 | < \dfrac{1}{C}, | \Im(z)| < \dfrac{\epsilon_0}{C} \right\rbrace .$$ 
We can consider a loop $ \gamma_E$ in:
$$ \hat{\Lambda}_E^{ \epsilon} = \lbrace (\theta, I) \in \tub( \mathbb{S}^1 \times \mathbb{R}), \hat{p}^{ \epsilon}(\theta, I) = p^{ \epsilon} \circ \hat{\kappa}^{-1}( \theta, I) = E \rbrace .$$
Thus the loop $ \hat{\kappa}^* \gamma_E := \hat{\gamma}_E$ is included in:
$$ \Lambda_E^{ \epsilon} = \lbrace (x, \xi) \in \tub(\mathbb{R}^2), p^{ \epsilon}(x, \xi) =E \rbrace .$$
And the following action integral is well-defined:
$$ \dfrac{1}{2 \pi} \int_{ \hat{\gamma}_E} \xi dx  = \dfrac{1}{2 \pi} \int_{ \gamma_E} I d \theta .$$
This explains why we can express the first term in the asymptotic expansion of eigenvalues of the operator $P_{ \hbar}^{ \epsilon}$ in terms of the action integral $ \int \xi dx $. \\

\noindent We want to quantize the complex canonical transformation $ \hat{\kappa}$ of the form:
\begin{align*}
\hat{\kappa} : \Vois( \Lambda_E^{ \epsilon}, \tub( \mathbb{R}^2)) & \longrightarrow \Vois( \hat{\Lambda}_E^{ \epsilon}, \tub( \mathbb{S}^1 \times \mathbb{R})), \\
(x, \xi) & \longmapsto ( \theta, I),
\end{align*}
where $( x, \xi)$ and $( \theta, I)$ denotes the complex coordinates. However, according to the proof of Theorem \ref{theo_cas_simple}, we know that there exists a complex canonical transformation:
\begin{align*}
\kappa : \Vois( \hat{\Lambda}_E^{ \epsilon}, \tub( \mathbb{S}^1 \times \mathbb{R})) & \longrightarrow \Vois( \tilde{I}=cst,(T^* \mathbb{S}^1)_{ \mathbb{C}}), \\
( \theta, I) & \longmapsto ( \tilde{\theta}, \tilde{I});
\end{align*} 
such that:
$$ \hat{p}^{ \epsilon} \circ \kappa^{-1}( \tilde{\theta}, \tilde{I}) = g^{ \epsilon}( \tilde{I}) .$$
 Consequently, instead of quantizing the transformation $ \hat{ \kappa}$, we can directly quantize the canonical transformation $ \kappa \circ \hat{\kappa}$. To do so, we follow the same steps as in the proof of Theorem \ref{theo_cas_simple}, thus we consider the following commutative diagram on the phase spaces:
\begin{displaymath}
\xymatrix{
\mathbb{R}^2 \subset \mathbb{C}^2 \ar[r]^{\kappa \circ \hat{\kappa}} \ar[d]_{ \kappa_{ \phi_1}} &
(T^* \mathbb{S}^1)_{ \mathbb{C}} \supset \mathbb{S}^1 \times \mathbb{R} \ar[d]^{ \overline{\kappa}_{ \phi_1}}\\
\overset{ \Lambda_{ \Phi_1} \subset \mathbb{C}^2}{ \Lambda_{ \Phi_2} \subset  \mathbb{C}^2}  & (T^* \mathbb{S}^1)_{\mathbb{C}} \supset \Lambda_{ \Phi_1} \ar[l]_{ \tilde{\kappa}^{-1}} } 
\end{displaymath}

\noindent We quantize the transformations as done previously and obtain the following diagram (with the notation of Propositions \ref{prop_H_Phi}, \ref{prop_quantif_A} and \ref{prop_B*_Egorov}):
\begin{displaymath}
    \xymatrix{
        L^2( \mathbb{R}) \ar[r]^{U_0} \ar[d]_{T_{ \phi_1}} & L^2_{\Flo}( \mathbb{S}^1) \ar[d]^{B^*} \\
        \overset{H(\mathbb{C}, \Phi_1)}{H( \pi(U), \Phi_2)} \ar[r]_{A}       & H_{\Flo}( \pi(V), \Phi_1) }
\end{displaymath}

\noindent \texttt{Notation:} $ \tilde{P}^{ \epsilon}_{ \hbar} = T_{ \phi_1} \circ P^{ \epsilon}_{ \hbar} \circ T_{ \phi_1}^*: H( \mathbb{C}, \Phi_2) \longrightarrow H( \mathbb{C}, \Phi_2)$.
\\

\noindent We can sum up what we have done so far by the following proposition. The proof of this result uses the same iterative procedure as in Proposition \ref{prop_tilde(UPU)=S}. 

\begin{prop}
There exists a unitary operator $ \tilde{U}: H( \pi(U), \Phi_2) \longrightarrow L^2_{\Flo}( \mathbb{S}^1)$ such that microlocally:
$$ \tilde{U} \tilde{P}^{ \epsilon}_{ \hbar} = g^{ \epsilon}_{ \hbar} \left( \dfrac{\hbar}{i}  \dfrac{\partial}{\partial \tilde{\theta}} \right) \tilde{U} + \mathcal{O}(\hbar^{ \infty}),$$
where $ \pi(U)$ is a suitable neighbourhood as in Proposition \ref{prop_tilde(UPU)=S} and where $g^{ \epsilon}_{ \hbar}$ is an analytic function admitting an asymptotic expansion in powers of $\hbar$, depending smoothly on $ \epsilon$ and such that its first term $g^{ \epsilon}_0$ is the inverse of the action integral $  \frac{1}{2 \pi} \int_{ \hat{\gamma}_E} \xi dx$.
\end{prop}

In order to determine the spectrum of the operator $P^{ \epsilon}_{ \hbar}$, we use the same arguments as in the proof of Theorem \ref{theo_cas_simple} and therefore two Grushin problems, one for the operator $S^{ \epsilon}_{ \hbar}-z$ and the other one for the operator $\tilde{P}^{ \epsilon}_{ \hbar}-z$.

\section{Application to $ \mathcal{P} \mathcal{T}$-symmetric pseudo-differential operators}
$ \mathcal{P} \mathcal{T}$-symmetric operators are used as an alternative to selfadjoint operators in quantum mechanics and an interesting question about any such operator is whether or not its spectrum is real (see \cite{Bender}). In the case of perturbations of pseudo-differential operators, Naima Boussekkine and Nawal Mecherout proved in \cite{BM} that $ \mathcal{P} \mathcal{T}$-symmetric perturbation of a semi-classical Schrödinger operator with a real-valued single well potential have real spectrum. Then Naima Boussekkine, Nawal Mecherout, Thierry Ramond and Johannes Sj{\"o}strand proved in \cite{BMRS} that in the case of a double well potential for an exponentially small perturbation of Schrödinger operator, this operator also has real spectrum. They also showed that for non-small perturbations of Schrödinger operator, the spectrum can become complex. \\

First, recall the definition of a $ \mathcal{P} \mathcal{T}$-symmetric operator (see for example \cite{BM} or \cite{BMRS}): we denote by $ \mathcal{P}$ the parity operator and by $ \mathcal{T}$ the time-reversal operator defined by:
\begin{align*}
\mathcal{P}: L^2( \mathbb{R}) & \longrightarrow L^2( \mathbb{R}) \quad & \mathcal{T}: L^2( \mathbb{R}) & \longrightarrow L^2( \mathbb{R}), \\
 u(x) & \longmapsto u(-x) \quad & u(x) & \longmapsto \overline{u(x)}.
\end{align*}
Let $P_{ \hbar}^{ \epsilon}$ be a pseudo-differential operator acting on $L^2( \mathbb{R})$ and depending smoothly on $ \epsilon$.

\begin{defi}
We said that the pseudo-differential operator $ P_{ \hbar}^{ \epsilon}$ is $ \mathcal{P} \mathcal{T}$-symmetric if $ [ P_{ \hbar}^{ \epsilon}, \mathcal{P} \mathcal{T} ] = 0 $.
\end{defi}

\begin{theoa}
Let $P^{ \epsilon}_{ \hbar}$ be a pseudo-differential operator depending smoothly on a small parameter $ \epsilon$, acting on $L^2( \mathbb{R})$ and let $E_0 \in \mathbb{R}$ such that they satisfy the assumptions (A) to (E), consequently the operator $P_{ \hbar}^{ \epsilon}$ is of the form:
$$ P^{ \epsilon}_{ \hbar}(x, \hbar D_x) = F^{ \epsilon}_{ \hbar}(x, \hbar D_x) + i \epsilon Q^{ \epsilon}_{ \hbar}(x, \hbar D_x) .$$
Moreover, we assume that $P_{ \hbar}^{ \epsilon}$ is $ \mathcal{P} \mathcal{T}$-symmetric. Let, for $ \epsilon_0>0$ a sufficiently small fixed real number:
$$ R_{C, \epsilon_0} = \left\lbrace z \in \mathbb{C}; | \Re (z) -E_0| <  \dfrac{1}{C}, | \Im(z) | < \dfrac{\epsilon_0}{C} \right\rbrace \quad \text{where $C > 0$ is a constant} .$$
Then the spectrum of the operator $P^{ \epsilon}_{ \hbar}$ in the rectangle $R_{C, \epsilon_0}$ is real for $0 \leq \epsilon < \epsilon_0$. \\
Besides, the spectrum $ \sigma(P_{ \hbar}^{ \epsilon})$ in the rectangle $R_{C, \epsilon_0}$ is given by Theorem \ref{theo_result}, thus for $0 \leq \epsilon < \epsilon_0$, we have:
$$ \sigma(P^{ \epsilon}_{ \hbar}) \cap R_{C, \epsilon_0} = \lbrace g^{ \epsilon}_{ \hbar}(\hbar k) , k \in \mathbb{Z} \rbrace \cap R_{C, \epsilon_0} + \mathcal{O}(\hbar^{ \infty}), $$
where $g^{ \epsilon}_{ \hbar}$ is an analytic function admitting an asymptotic expansion in powers of $\hbar$, depending smoothly on $ \epsilon$ and such that its first term $g_0^{ \epsilon}$ in the asymptotic expansion is the inverse of the action coordinate $ \frac{1}{2 \pi} \int_{ \gamma_E} \xi dx$.
\end{theoa}

\begin{proof}
According to Theorem \ref{theo_result}, we know that the spectrum of the operator $P_{ \hbar}^{ \epsilon}$ in the rectangle $R_{C, \epsilon_0}$ is given by:
$$ \sigma(P^{ \epsilon}_{ \hbar}) \cap R_{C, \epsilon_0} = \lbrace g^{ \epsilon}_{ \hbar}(\hbar k) , k \in \mathbb{Z} \rbrace \cap R_{C, \epsilon_0} + \mathcal{O}(\hbar^{ \infty}),$$
where $g^{ \epsilon}_{ \hbar}$ is an analytic function admitting an asymptotic expansion in powers of $\hbar$, depending smoothly on $ \epsilon$ and the first term $g^{ \epsilon}_0$ is the inverse of the action coordinate $ \frac{1}{2 \pi} \int_{ \gamma_E} \xi dx$. This means that the eigenvalues are along a curve up to $ \mathcal{O}( \hbar^{ \infty})$. \\
Moreover, since $P_{ \hbar}^{ \epsilon}$ is $ \mathcal{P} \mathcal{T}$-symmetric, we have $ \mathcal{P} \mathcal{T}( P_{ \hbar}^{ \epsilon} - z) = (P_{ \hbar}^{ \epsilon} - \overline{z}) \mathcal{P} \mathcal{T}$. Therefore the spectrum $ \sigma( P_{ \hbar}^{ \epsilon})$ is symmetric with respect to the real axis.  \\
If we choose an eigenvalue in the spectrum $ \sigma(P_{ \hbar}^{ \epsilon})$, then the symmetric of this eigenvalue must also be in the spectrum. Yet, the distance between the real parts of two eigenvalues is of order $ \mathcal{O}( \hbar)$, therefore the symmetric of an eigenvalue has the same real part as the eigenvalue itself. Therefore the symmetric of an eigenvalue is the eigenvalue itself, \textit{i.e.} the spectrum is real. As a result, we obtain that $ \sigma(P_{ \hbar}^{ \epsilon}) \cap R_{C, \epsilon_0}$ is real.
\end{proof}

\begin{remark}
We recover the result of Naima Boussekkine and Nawal Mecherout in \cite{BM} by using this theorem for any real number $E_0$ satisfying Hypothesis (E) (\textit{i.e.} for non-critical point $E_0$) and the result of Michael Hitrik in \cite{MR2101278} for critical points $E_0$.
\end{remark}

\section{Numerical illustrations}
\noindent In this section, we illustrate our result for several differential operators. The following plots have been obtained with the numerical computation software Scilab.

\subsection{Operators acting on $L^2( \mathbb{S}^1)$}
Let $ \alpha \in \mathbb{R}^*$. In this section, we deal with differential operators $P^{ \epsilon}$ acting on $L^2( \mathbb{S}^1)$ of the form:
$$ P^{ \epsilon}( \theta, \hbar D_{ \theta}) = \alpha \hbar D_{ \theta} + i \epsilon Q( \theta, \hbar D_{ \theta}) ,$$
where the symbol $q( \theta, I)$ associated with the operator $Q( \theta, \hbar D_{ \theta})$ is an analytic function on $ \mathbb{S}^1 \times \mathbb{R}$ which does not depend on the semi-classical parameter $ \hbar$. \\
To implement this type of operators and determine their spectra by numerical methods, we follow these three steps:
\begin{enumerate}
\item[1.] notice that the family $(e_l)_{l \in \mathbb{Z}} = (e^{i l \theta})_{l \in \mathbb{Z}}$ is an orthonormal basis of $L^2( \mathbb{S}^1)$, therefore we can define the operator $P^{ \epsilon}$ by its action on the basis, so we obtain an infinite matrix $ \mathcal{P}^{ \epsilon}$;
\item[2.] we choose an integer $N \geq 1$ and we restrict the matrix $ \mathcal{P}^{ \epsilon}$ to a matrix $ \mathcal{P}^{ \epsilon}_{2N+1}$ of size $(2N+1) \times (2N+1)$ by choosing to only consider the action of the operator $P^{ \epsilon}$ on the functions $(e_l)_{-N \leq l \leq N}$;
\item[3.] we compute the spectrum of $ \mathcal{P}^{ \epsilon}$ with the function \verb+spec+ of Scilab.
\end{enumerate}
Then, to compare the numerical spectrum with our result, we determine an approximate of the function $g^{ \epsilon}( \tilde{I})$ (which gives the exact spectrum) by considering the average in $ \theta$ of the symbol $p^{ \epsilon}( \theta, I) := \alpha I + i \epsilon q( \theta, I)$. \\
We obtain the following plots by using the parameters:
\begin{enumerate}
\item[1.] $N=66$;
\item[2.] $\hbar= \dfrac{1}{N}$;
\item[3.] $ \alpha = 1$;
\item[4.] $ \epsilon = \hbar^{ \delta}$ with $ \delta = \dfrac{1}{2}$.
\end{enumerate}

\begin{figure}[t!]
\includegraphics[width=13cm]{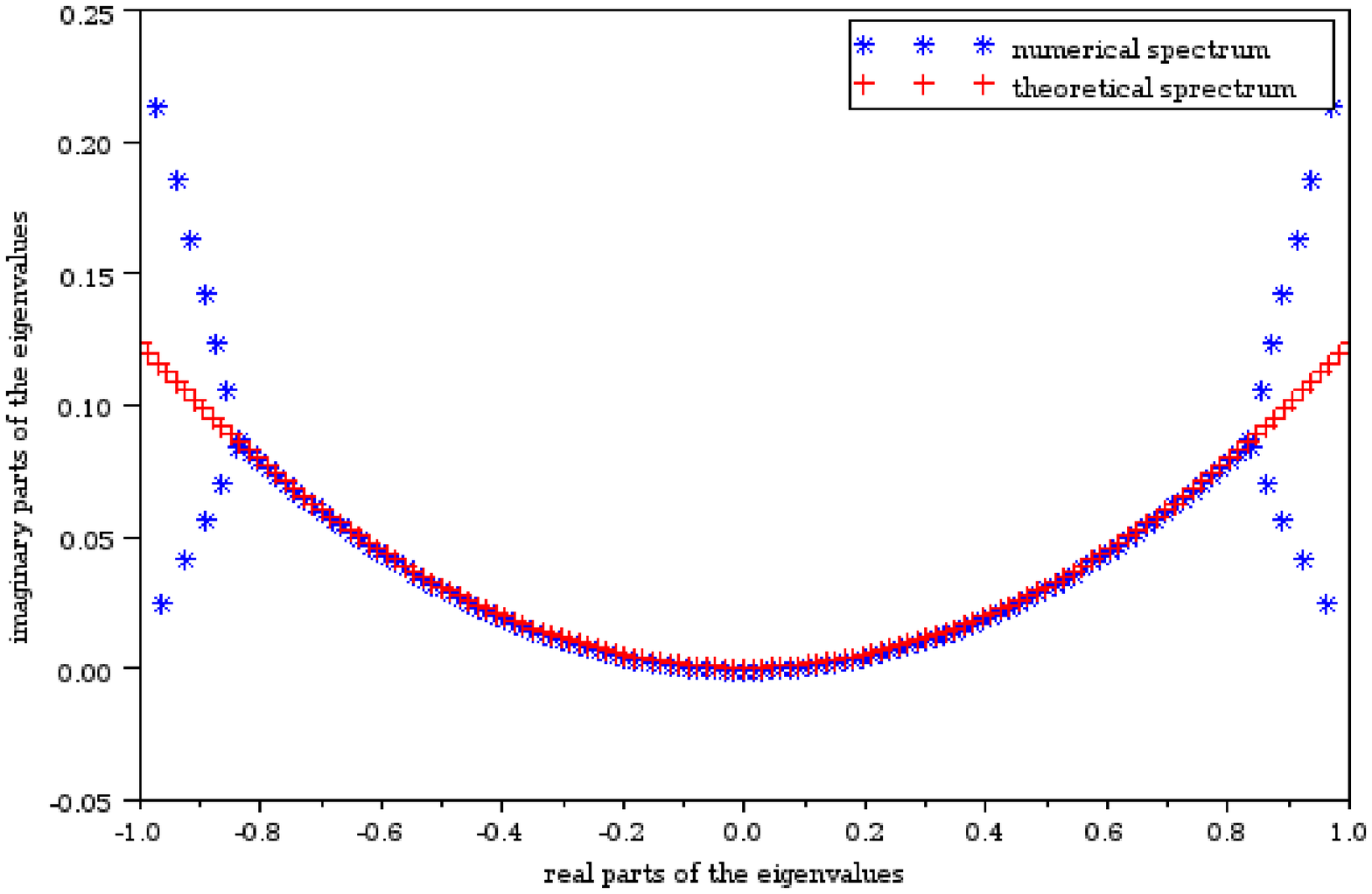}
\caption{$p^{ \epsilon}( \theta, I) = \alpha I + i \epsilon ( \cos \theta + I^2)$.}
\end{figure}

\begin{figure}[b!]
\includegraphics[width=13cm]{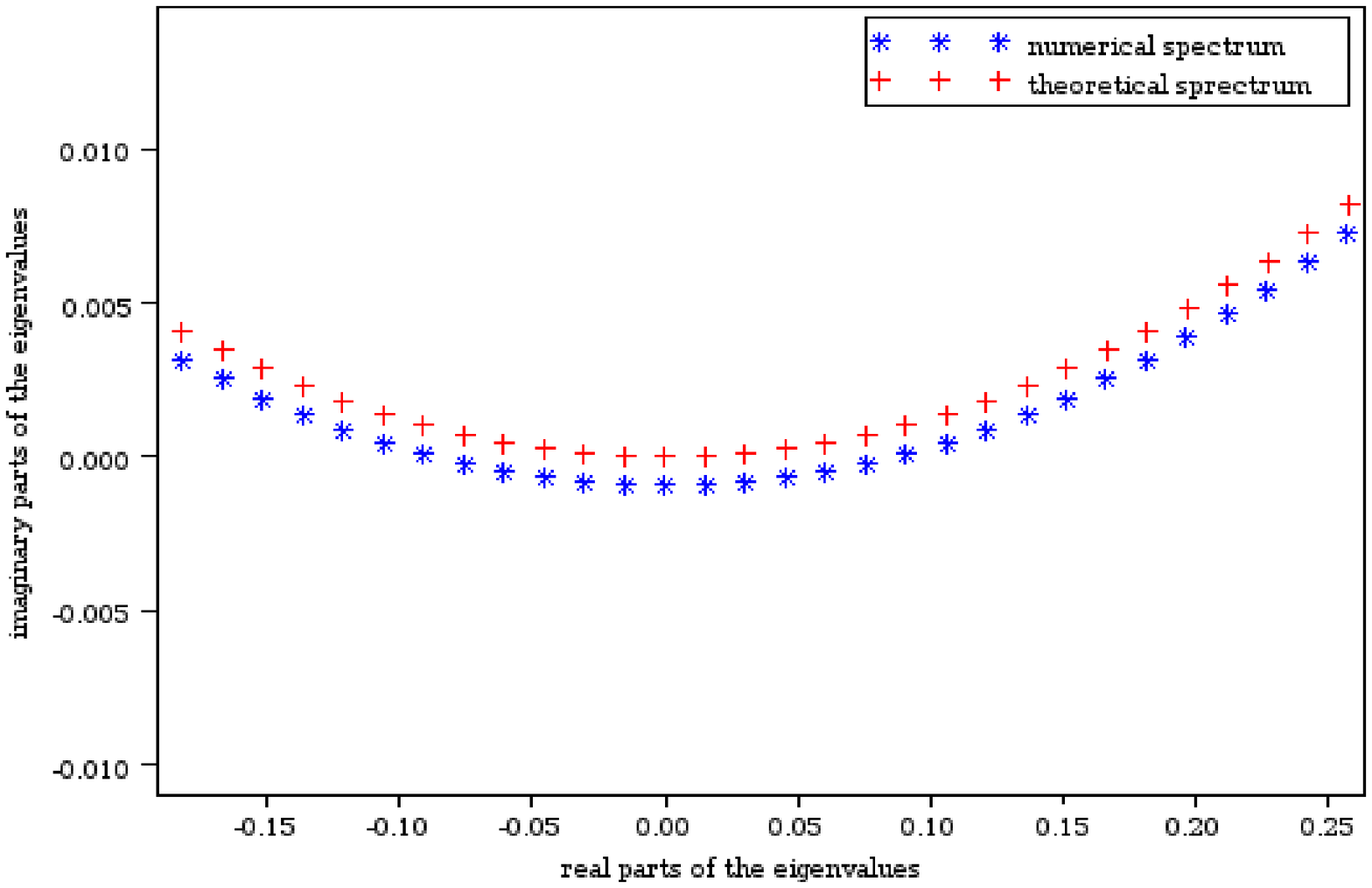}
\caption{$p^{ \epsilon}( \theta, I) = \alpha I + i \epsilon ( \cos \theta + I^2)$.}
\end{figure}

\begin{figure}[t!]
\includegraphics[width=13cm]{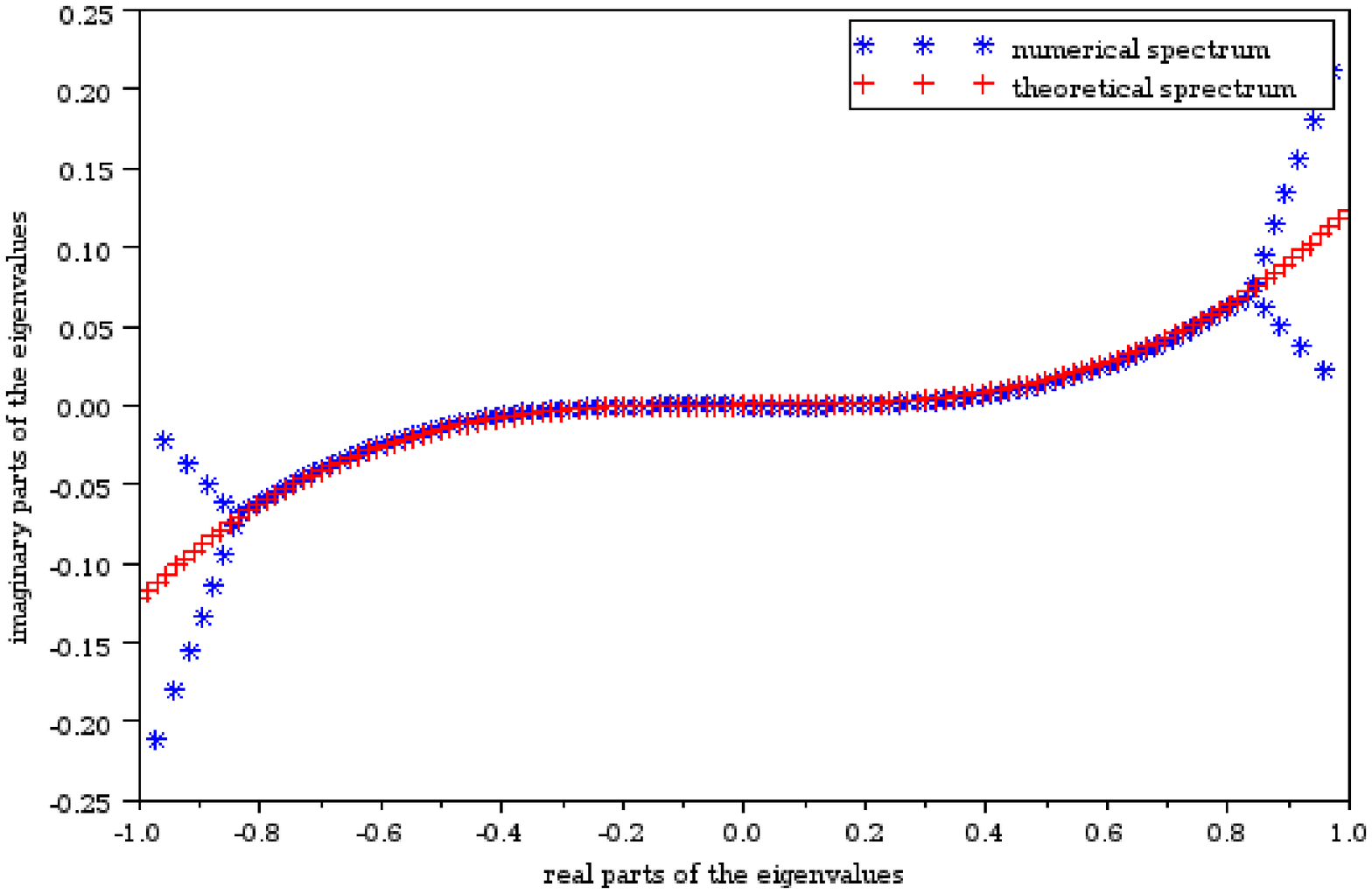}
\caption{$p^{ \epsilon}( \theta, I) = \alpha I + i \epsilon ( \cos \theta + I^3)$.}
\end{figure}

\begin{figure}[b!]
\includegraphics[width=13cm]{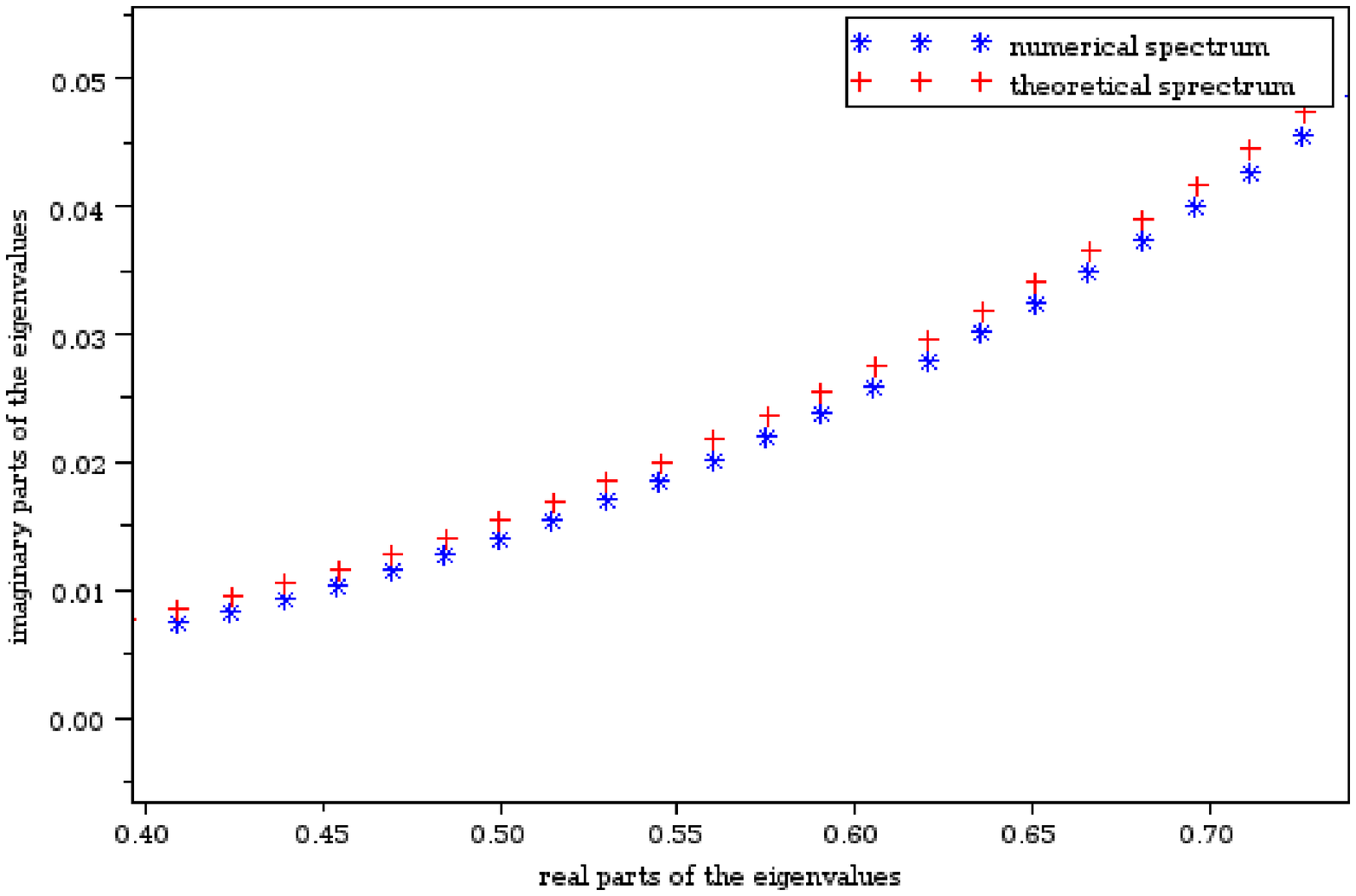}
\caption{$p^{ \epsilon}( \theta, I) = \alpha I + i \epsilon ( \cos \theta + I^3)$.}
\end{figure}

\newpage 
\subsection{Operators acting on $L^2( \mathbb{R})$}
In this section, we deal with differential operators $P^{ \epsilon}$ acting on $L^2( \mathbb{R})$ of the form:
$$ P^{ \epsilon}( x, \hbar D_x) = P_0(x, \hbar D_x) + i \epsilon Q(x, \hbar D_x) ,$$
where $P_0(x, \hbar D_x) = x^2 + (\hbar D_x)^2$ is the harmonic oscillator and where $q(x, \xi)$ the symbol associated with the operator $Q(x, \hbar D_x)$ is a polynomial function in $x$ and $ \xi$, which does not depend on the semi-classical parameter $ \hbar$.\\
To implement this type of operator, we consider the following space.

\begin{defi}[Fock space]
The Fock space, denoted by $ \mathcal{F}$, is the set of holomorphic functions $f(z)$ on $ \mathbb{C}$ satisfying:
$$ \dfrac{1}{\pi} \int_{ \mathbb{C}} |f(z)|^2 e^{-|z|^2/\hbar} L(dz) < + \infty .$$
\end{defi}

\noindent \texttt{Notation:} $ \langle ,  \rangle$ is the scalar product on $ \mathcal{F}$ defined for all $u, v \in \mathcal{F}$ by:
$$ \langle u, v \rangle = \dfrac{1}{ \pi} \int_{ \mathbb{C}} \overline{ u(z)} v(z) e^{-|z|^2/\hbar} L(dz) .$$

\noindent We can show that, for $ \alpha \in \mathbb{N}$, the family $( \zeta_{ \alpha})_{ \alpha \in \mathbb{N}}$, where:
$$ \zeta_{ \alpha}(z)= \dfrac{z^{ \alpha}}{\sqrt{ \hbar^{ \alpha +1} \alpha !}},$$
is an orthonormal basis of $ \mathcal{F}$. Recall the definition of the Bargmann transform associated with the Fock space.

\begin{defi}[Bargmann transform]
Let $u \in L^2( \mathbb{R})$, we define the Bargmann transform of $u$, for $z \in \mathbb{C}$, by:
$$ (Tu)(z) = \int_{ \mathbb{R}} e^{-(z^2 - 2 \sqrt{2} xz + x^2)/(2\hbar)} u(x) dx .$$
This transform sends $L^2( \mathbb{R})$ to the Fock space $ \mathcal{F}$.
\end{defi}

\noindent To determine the spectrum of the operator $P^{ \epsilon}$ by numerical methods, we follow these three steps:
\begin{enumerate}
\item[1.] we compute $T P^{ \epsilon} T^{-1}$ by using creation and annihilation operators and we define the operator $TP^{ \epsilon} T^{-1}$ by its action on the basis $( \zeta_{ \alpha})_{ \alpha \in \mathbb{N}}$, so we obtain an infinite matrix $ \mathcal{P}^{ \epsilon}$;
\item[2.] we choose an integer $N \geq 1$ and we restrict the matrix $ \mathcal{P}^{ \epsilon}$ to a matrix $ \mathcal{P}^{ \epsilon}_{N+1}$ of size $(N+1) \times (N+1)$ by choosing to only consider the action of the operator $TP^{ \epsilon} T^{-1}$ on the functions $(\zeta_{ \alpha})_{0 \leq \alpha \leq N}$;
\item[3.] we compute the spectrum of $ \mathcal{P}^{ \epsilon}$ with the function \verb+spec+ of Scilab.
\end{enumerate}
Then, to compare the numerical spectrum with our result, we determine an approximate of the function $g^{ \epsilon}$ (which gives the exact spectrum) by giving explicit action-angle coordinates for the harmonic oscillator and by computing an approximate to order $ \epsilon$ of the function $g^{ \epsilon}$ (by averaging $q$). We denote this approximate by $ \tilde{g}^{ \epsilon}$. \\
We compare the numerical result with the approximate spectrum obtained by using our theorem (\textit{i.e.} $ \tilde{g}^{ \epsilon}( \hbar k)$ with $k \in \mathbb{Z}$) and with the approximate spectrum obtained by using the spectrum of the harmonic oscillator (\textit{i.e.} $\tilde{g}^{ \epsilon}(\hbar(2k+1))$ with $k \in \mathbb{Z}$). We observe that the approximate spectrum obtained via the spectrum of the harmonic oscillator is better than the one obtained with our result, because it takes into account the Maslov index.\\

\noindent We obtain the following plots by using the parameters:
\begin{enumerate}
\item[1.] $N=66$;
\item[2.] $\hbar = \dfrac{1}{N}$;
\item[3.] $ \epsilon = \hbar^{ \delta}$ with $ \delta = \dfrac{1}{2}$.
\end{enumerate}

\begin{figure}[!!b]
\includegraphics[width=13cm]{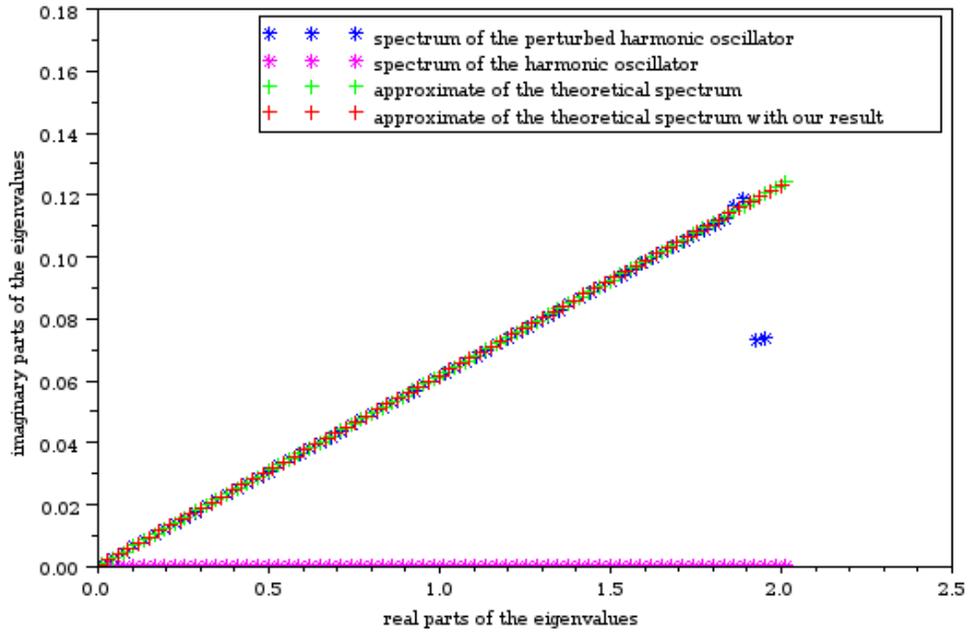}
\caption{$p^{ \epsilon}(x, \xi) = x^2 + \xi^2 + i \epsilon x^2$.}
\end{figure}

\begin{figure}[!!t]
\includegraphics[width=13cm]{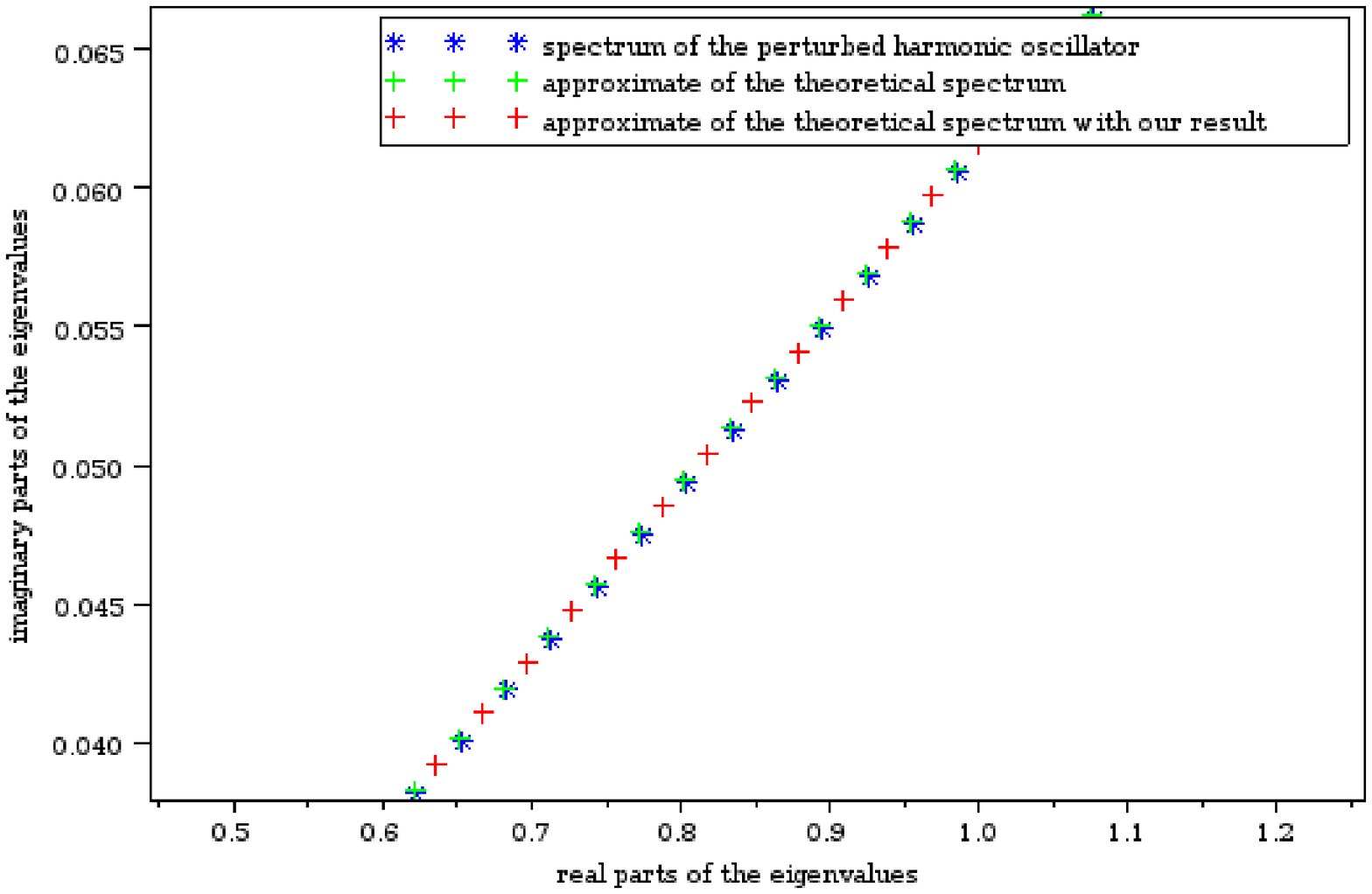}
\caption{$p^{ \epsilon}(x, \xi) = x^2 + \xi^2 + i \epsilon x^2$.}
\end{figure}

\begin{figure}[!!b]
\includegraphics[width=13cm]{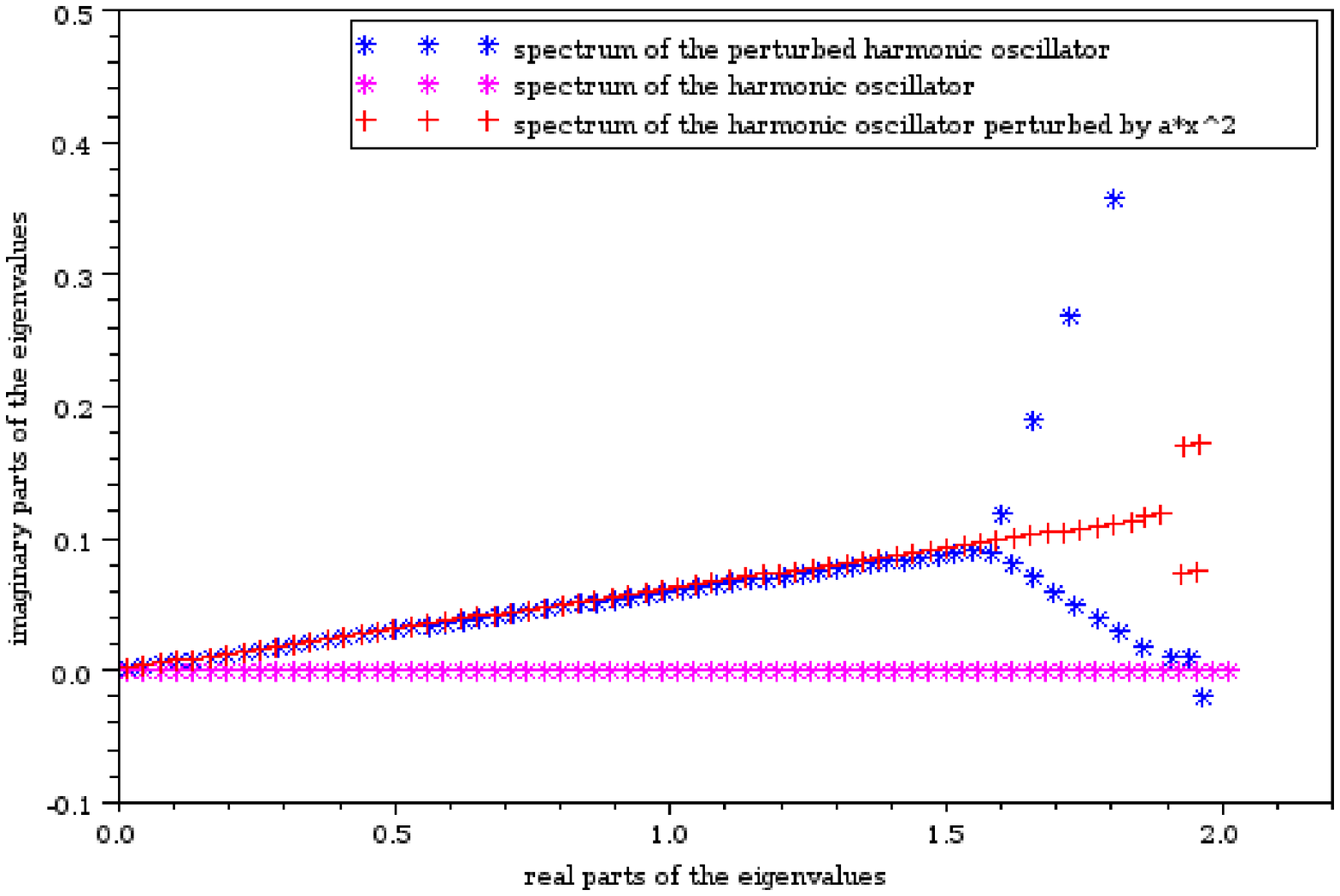}
\caption{$p^{ \epsilon}(x, \xi) = x^2 + \xi^2 + i \epsilon (x^2+x^3)$.}
\end{figure}

\begin{figure}[!!t]
\includegraphics[width=13cm]{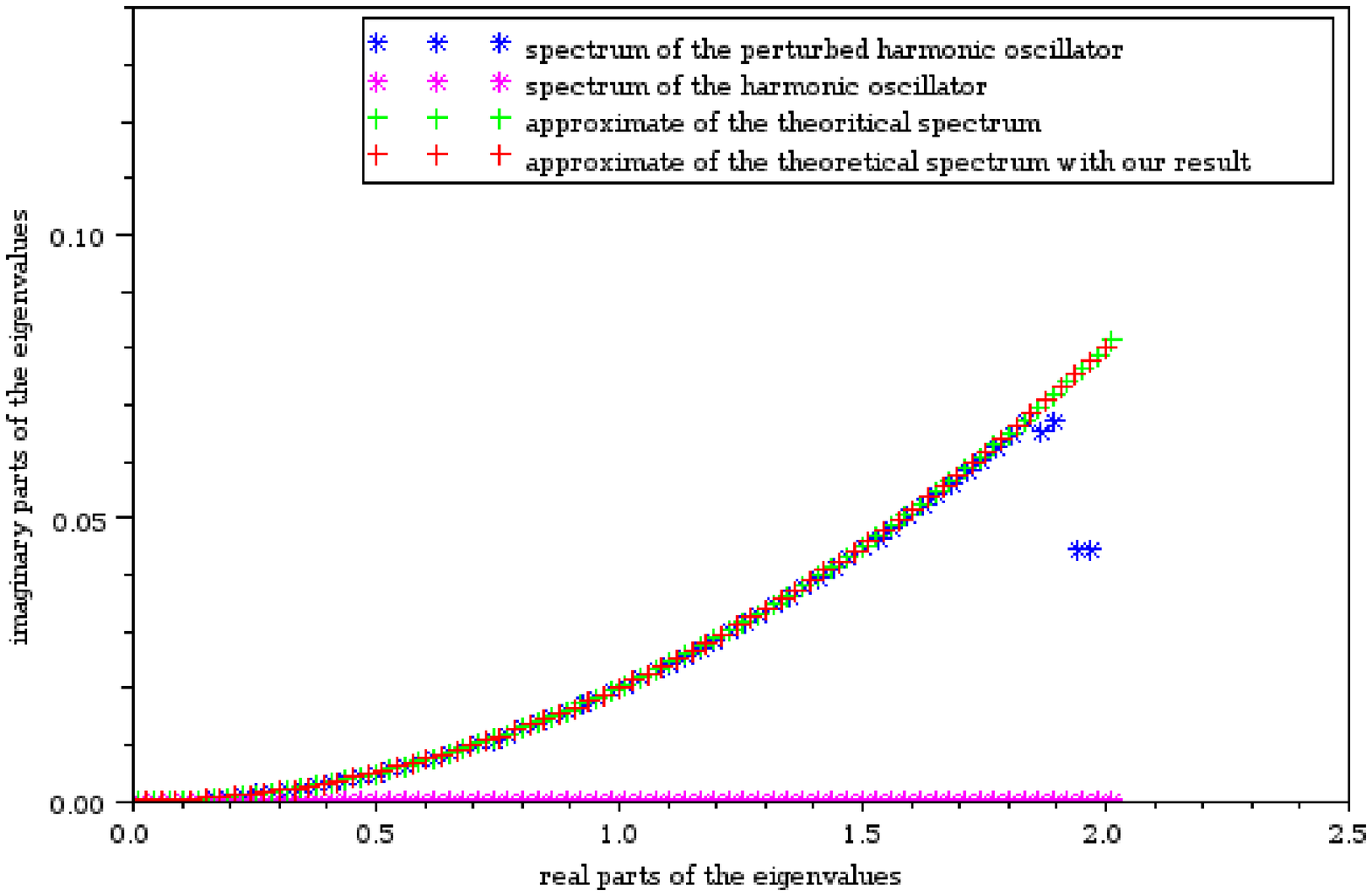}
\caption{$p^{ \epsilon}(x, \xi) = x^2 + \xi^2 + i \epsilon x^4$.}
\end{figure}

\begin{figure}[!!b]
\includegraphics[width=13cm]{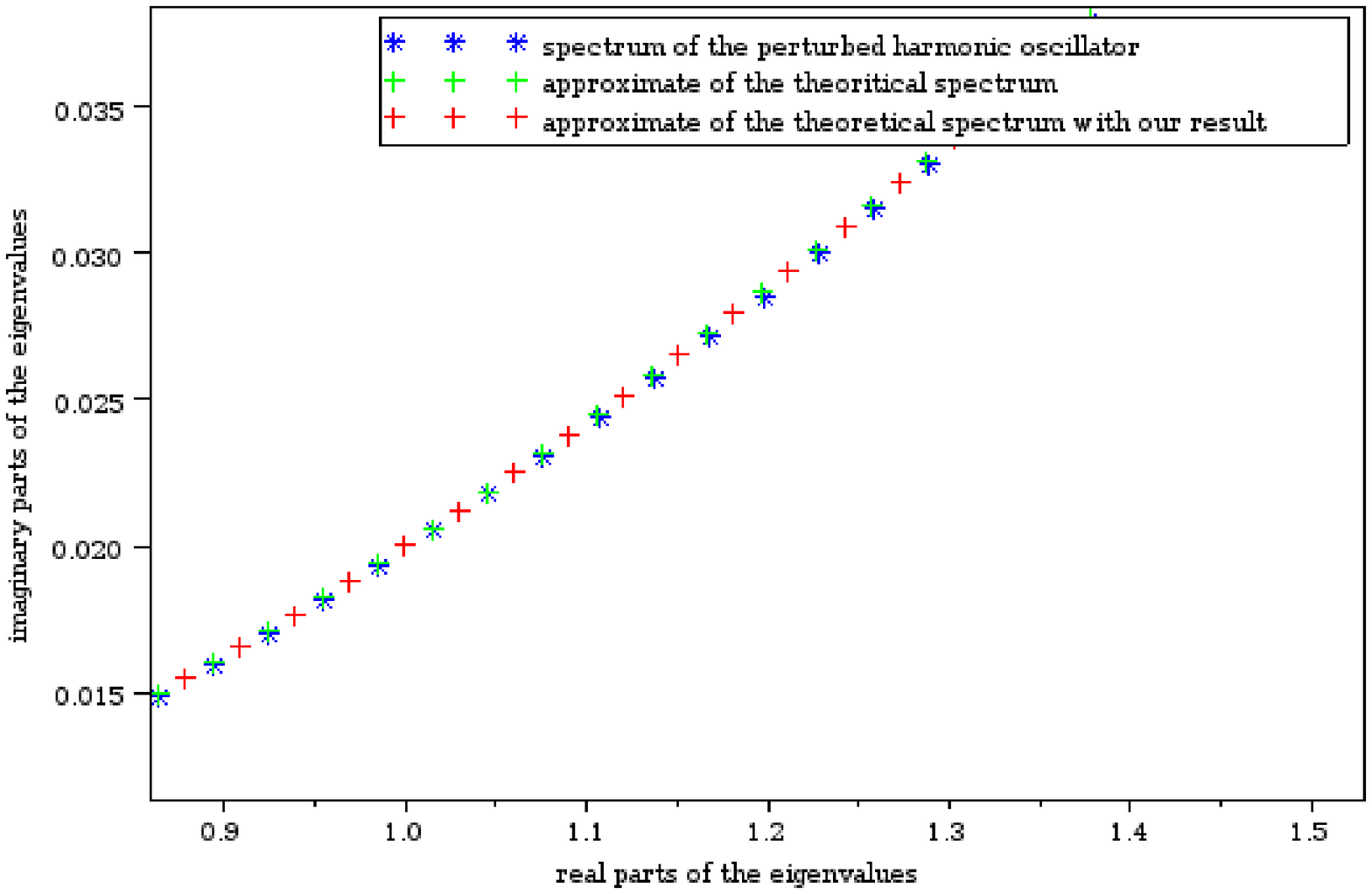}
\caption{$p^{ \epsilon}(x, \xi) = x^2 + \xi^2 + i \epsilon x^4$.}
\end{figure}

\newpage

\renewcommand{\thesection}{\Alph{section}}
\setcounter{section}{0}

\nocite{*}
\bibliographystyle{amsalpha}
\bibliography{biblio}

\end{document}